\documentclass[11pt]{amsart}
\usepackage[a4paper,centering,left=3.2cm,right=3.2cm]{geometry}
\usepackage{amsfonts,amscd,amssymb,amsmath,amsthm,mathrsfs, multirow,xcolor,lscape,longtable,pbox,lipsum,afterpage,graphicx,threeparttable, comment}
\usepackage[thinlines]{easytable}
\usepackage{enumerate}
\usepackage{todonotes}
\usepackage[colorlinks,linkcolor=blue,anchorcolor=blue,citecolor=blue,backref=page]{hyperref}
\usepackage{hyperref}

\allowdisplaybreaks
\newtheorem{proposition}{Proposition}[section]
\newtheorem{lemma}[proposition]{Lemma}
\newtheorem{corollary}[proposition]{Corollary}
\newtheorem{theorem}[proposition]{Theorem}

\theoremstyle{definition}

\newtheorem{remark}[proposition]{Remark}

\newtheorem*{assumption}{Assumption}
\numberwithin{equation}{section}

\usepackage[all,cmtip]{xy}
\usepackage{tikz}
\usetikzlibrary{arrows} 
\usetikzlibrary{matrix}
\usetikzlibrary{calc}

\DeclareMathOperator{\mdeg}{mdeg}

\makeatletter
\newcommand{\customlabel}[2]{%
   \protected@write \@auxout {}{\string \newlabel {#1}{{#2}{\thepage}{#2}{#1}{}} }%
   \hypertarget{#1}{#2}%
}
\makeatother

\usepackage{mathtools}
\def\csgp{\mathrel{
\mathrlap{\trianglelefteq}{\raisebox{1pt}{$\blacktriangleleft$}}
}}

\begin{document}
\title{A CFSG-free explicit Jordan's theorem over arbitrary fields}
\date{\today}

\author{Jitendra Bajpai}
\address{Department of Mathematics, Christian-Albrechts-University of Kiel, 24118 Kiel, Germany}
\email{jitendra@math.uni-kiel.de}

\author{Daniele Dona}
\address{Hun-Ren Alfr\'ed R\'enyi Institute of Mathematics, 1053 Budapest, Hungary}
\email{dona@renyi.hu}


\begin{abstract}
We prove a version of Jordan's classification theorem for finite subgroups of $\mathrm{GL}_{n}(K)$ that is at the same time quantitatively explicit, CFSG-free and valid for arbitrary $K$. This is the first proof to satisfy all three properties at once. Our overall strategy follows Larsen and Pink~\cite{LP11}, with explicit computations based on techniques developed by the authors and Helfgott~\cite{BDH21, BDH24}, particularly in relation to dimensional estimates.
\end{abstract}

\maketitle
\tableofcontents  

\section{Introduction}\label{se:intro}

Results about the structure of subgroups of $\mathrm{GL}_{n}(\mathbb{C})$ have been known for a long time, at least since Jordan proved the following result~\cite[Thm.~40]{Jor78}.

\begin{theorem}[Jordan's theorem]\label{th:jordan}
Let $\Gamma$ be a finite subgroup of $\mathrm{GL}_{n}(\mathbb{C})$. Then there is a normal abelian subgroup $A\unlhd\Gamma$ of index bounded by a constant $J(n)$ depending only on $n$.
\end{theorem}

Since then, $J(n)$ has been bounded explicitly. A bound of the form $e^{O(n^2/\log n)}$ is given in~\cite[Thm.~14.12]{Isa76}, based on ideas of Frobenius and Blichfeldt, and it does not use the Classification of Finite Simple Groups (CFSG). With the aid of CFSG, Weisfeiler proved the bound $(n+1)!n^{O(\log n)}$ in~\cite{Wei84}, improved by the same author in \cite{Wei12} to $(n+2)!$ for $n\geq 64$, and by Collins in~\cite[Thm.~A]{Collins2007} to $(n+1)!$ for $n\geq 71$, which is tight in general. We refer the reader for an exposition on Jordan's theorem to a recent survey by Breuillard~\cite{Breuillard2023}.

Theorem~\ref{th:jordan} is false if we replace $\mathbb{C}$ with a field of positive characteristic. Nevertheless, there exist results about the structure of finite subgroups of $\mathrm{GL}_{n}(K)$ that generalize Jordan's theorem. Here we prove one such result.

\begin{theorem}\label{th:main}
Let $K$ be any field, and let $\Gamma$ be a finite subgroup of $\mathrm{GL}_{n}(\overline{K})$. Then there are $\Gamma_{3}\unlhd\Gamma_{2}\unlhd\Gamma_{1}\unlhd\Gamma$, each of them normal inside $\Gamma$, such that
\begin{enumerate}[(a)]
\item\label{th:main-small} $|\Gamma/\Gamma_{1}|\leq J'(n):=n^{n^{2^{23}n^{10}}}$;
\item\label{th:main-lie} either $\Gamma_{1}=\Gamma_{2}$, or $\mathrm{char}(K)=p>0$ and $\Gamma_{1}/\Gamma_{2}$ is a product of finite simple groups of Lie type of characteristic $p$;
\item\label{th:main-ab} $\Gamma_{2}/\Gamma_{3}$ is abelian of size not divisible by $\mathrm{char}(K)$;
\item\label{th:main-p} either $\Gamma_{3}=\{e\}$, or $\mathrm{char}(K)=p>0$ and $\Gamma_{3}$ is a $p$-group.
\end{enumerate}
\end{theorem}

A version of Theorem~\ref{th:main} without any explicit expression for $J'(n)$ was proved by Larsen and Pink~\cite{LP11}, without relying on CFSG. With the use of CFSG, Weisfeiler~\cite{Wei84} showed that we can take $J'(n)=\lfloor(3n+9)/2\rfloor!n^{O(\log n)}$, improved in~\cite{Wei12} to $J'(n)=(n+2)!$ for $n\geq 64$, and in~\cite{Col08} by Collins to $J'(n)=(n+1)!$ for $n\geq 71$ and $\mathrm{char}(K)\nmid n+2$. Our goal is to have at the same time an explicit $J'(n)$ in the statement and a CFSG-free proof. The present paper is the first to have both properties for an arbitrary field $K$.

Our interest in the question stems from our previous work on {\em dimensional estimates}. These tools were developed first in~\cite[\S 4]{LP11} in a non-explicit form, and then used in several papers in the context of Babai's conjecture. Most recently, the authors in joint work with Helfgott gave explicit dimensional estimates in order to achieve sharper diameter bounds for untwisted classical groups~\cite{BDH21,BDH24}, and a natural question was whether the rest of the techniques of~\cite{LP11} could be made equally explicit. In the present paper, we follow the procedure in~\cite{LP11}, sharpening and cleaning the route taken by them through the strategies involved in the proofs of~\cite{BDH21, BDH24}.

\subsection{Outline of the strategy}

The first three sections collect some preliminary facts: Section~\ref{se:var} concerns varieties, Section~\ref{se:linalg} deals with linear algebraic groups in general, and Section~\ref{se:almsim} focuses on almost simple groups. Some of the definitions and properties are standard, some are taken from~\cite{BDH21,BDH24}, and some are new although in line with the spirit of those papers.

Section~\ref{se:dimest} deals with dimensional estimates. A {\em dimensional estimate} is a bound of the form $|A\cap V(K)|\leq C|A^{C}|^{\dim(V)/\dim(G)}$, where $G$ is an algebraic group over $K$, $V\subseteq G$ is a subvariety, $A\subseteq G(K)$ is a finite subset, and $C$ is some constant depending only on the data of $G$ and $V$ (but not on $A$ and $K$). Such an estimate appeared first in~\cite[Thm.~4.2]{LP11}, with $A=\Gamma$ a subgroup and with a non-explicit $C$. The bounds in~\cite[Thm.~4.4]{BDH21} and~\cite[Thm.~1.1]{BDH24} have instead an explicit $C$ and hold for $A$ a generating set of $G(K)$. Our task in this section is to show that the assumption $\langle A\rangle=G(K)$ can be weakened, so that we may have estimates for $A=\Gamma$ with explicit $C$. The section plays the role of~\cite[\S 4]{LP11} through its main result (Theorem~\ref{th:dimest}), and of~\cite[\S 6]{LP11} through its corollary on centralizers (Corollary~\ref{co:estcentr}).

The goal of Section~\ref{se:findgf} is to prove an explicit version of \cite[Thm.~0.5]{LP11}, following the path laid out in~\cite[\S\S 7--11]{LP11}. Its main result (Theorem~\ref{th:05}) shows that, for any $G$ almost simple and any $\Gamma\leq G(\overline{K})$, either $[G^{F},G^{F}]\leq\Gamma\leq G^{F}$ for some appropriate endomorphism $F$, where $G^{F}$ is a group of Lie type and its commutator is simple, or $\Gamma$ is trapped in some substructure: either $|\Gamma|$ is bounded in terms of the rank of $G$, or $|\Gamma|\leq H(\overline{K})$ for some proper subgroup $H\lneq G$ of smaller dimension and bounded degree.

The path of Section~\ref{se:findgf} is articulated in several steps. Starting from the assumption that $\Gamma$ is not trapped as above, we first find regular unipotent elements in $\Gamma$, thus incidentally proving that $\mathrm{char}(K)$ must be positive (Proposition~\ref{pr:run})\footnote{At this stage, a CFSG-free proof of Theorem~\ref{th:jordan} with explicit $J(n)$ is already within reach. We do not bother doing so, since~\cite[Thm.~14.12]{Isa76} already does that and with a better $J(n)$ than ours.}. Then we find a variety $V$ of minimal unipotent elements, representing the finite field $\mathbb{F}_{q}$ that ``correctly determines'' $\Gamma$ (Proposition~\ref{pr:findfield}): when at the end $[G^{F},G^{F}]\leq\Gamma\leq G^{F}$, $F$ will be either the Frobenius map with respect to $\mathbb{F}_{q}$ or a twist of that map.

For now $\mathbb{F}_{q}$ is a good model only for the minimal unipotent elements of $\Gamma$, meaning that $\Gamma\cap V(\overline{K})\simeq\mathbb{F}_{q}$ (as abelian groups). The final step of Section~\ref{se:findgf} is to prove that $\mathbb{F}_{q}$ is a good model for the whole $\Gamma$. We do so in two stages: first, in Propositions~\ref{pr:91base}--\ref{pr:9105} we achieve our goal under some conditions on $\Gamma$ (Assumption~\ref{ass-large1}--\ref{ass-h1}) and on $V$ ($\dim(V)=\mathrm{rk}(G)$); then, we use that partial case and stronger conditions on $\Gamma$ (Assumption~\ref{ass-large2}--\ref{ass-h2}) to complete the proof without any hypothesis on $V$. The general case is Theorem~\ref{th:05}.

In Section~\ref{se:final} we complete the proof of Theorem~\ref{th:main}. In rough terms, the case of $[G^{F},G^{F}]\leq\Gamma\leq G^{F}$ gives rise to~\eqref{th:main-lie}, whereas the case of $|\Gamma|$ small gives rise to~\eqref{th:main-small}. The descent from $\Gamma\leq G(\overline{K})$ to $\Gamma\leq H(\overline{K})$ with $\dim(H)<\dim(G)$ can be repeated until we reach either one of the other cases or $\dim(H)=0$ (in which case $|\Gamma|$ is small again, thanks to the bound on $\deg(H)$). Since $H$ is not necessarily almost simple, at every stage we need to take quotients by the unipotent radical and by the centre: the former is a $p$-group, whence~\eqref{th:main-p}, and the latter is abelian, whence~\eqref{th:main-ab}.

\section{Varieties}\label{se:var}

In this section, we collect basic properties about varieties, morphisms, and degrees.

\subsection{Basic nomenclature}\label{sss:verybasic}

We go over some standard terms, whose definition in the literature can vary.

A {\em variety}\footnote{For us a variety is affine and closed, not necessarily irreducible nor connected nor pure-dimensional.} $V$ in $n$-dimensional affine space $\mathbb{A}^{n}$ is defined by a set of $s$ equations of the form $P_{i}(x_{1},\ldots,x_{n})=0$ (for $1\leq i\leq s$), where all $P_{i}$ are polynomials and $s$ is any non-negative integer. $V$ is {\em defined over} a field $K$ if the coefficients of all $P_{i}$ belong to $K$. The {\em set of points $V(K)$} is
\begin{equation*}
V(K)=\{(k_{1},\ldots,k_{n})\in K^{n}:  P_{i}(k_{1},\ldots,k_{n})=0\;\;(1\leq i\leq s)\}.
\end{equation*}
Two varieties $V,W$ are equal if and only if the ideals generated by their defining polynomials inside the ring $\overline{K}[x_{1},\ldots,x_{n}]$ have the same radical, which by the Nullstellensatz holds if and only if $V(\overline{K})=W(\overline{K})$.

Let $V,W$ be defined by polynomials $\mathcal{P}=\{P_{i}\}_{i\leq s}$ and $\mathcal{Q}=\{Q_{j}\}_{j\leq t}$. The {\em (set-theoretic, or reduced) intersection} $V\cap W$ is the variety defined by $\mathcal{P}\cup\mathcal{Q}$, and the {\em union} $V\cup W$ is defined by $\{P_{i}Q_{j}\}_{i\leq s, j\leq t}$.

The {\em Zariski topology} is the topology whose closed sets are the sets $V(\overline{K})$ for all varieties $V$. The affine space $\mathbb{A}^{n}(\overline{K})$ is Noetherian under this topology. The {\em Zariski closure} $\overline{S}$ of a set $S\subseteq\mathbb{A}^{n}(\overline{K})$ is the smallest set of the form $V(\overline{K})$ containing $S$. We call $V$ itself the Zariski closure of $S$.

A variety $V$ is {\em irreducible} if it is not equal to any union $V_{1}\cup V_{2}$ with $V_{1}\not\subseteq V_{2}$ and $V_{2}\not\subseteq V_{1}$. Every $V$ can be uniquely decomposed into a finite union of irreducible varieties not contained in each other, called the {\em irreducible components} of $V$. The {\em dimension} $\dim(V)$ of an irreducible variety $V$ is the largest $d$ for which we can write a chain of irreducible proper subvarieties $V_{0}\subsetneq V_{1}\subsetneq\ldots\subsetneq V_{d}=V$. For $V$ non-irreducible, $\dim(V)$ is the largest of the dimensions of its irreducible components. A variety is {\em pure-dimensional} when all its components have the same dimension.

A {\em morphism} $f:\mathbb{A}^{n}\rightarrow\mathbb{A}^{m}$ defined over $K$ is an $m$-tuple of polynomials $f_{i}$ on $n$ variables whose coefficients belong to $K$. A morphism $f:X\rightarrow Y$ for $X\subseteq\mathbb{A}^{n},Y\subseteq\mathbb{A}^{m}$ is the restriction of a morphism $g:\mathbb{A}^{n}\rightarrow\mathbb{A}^{m}$ such that $g(x)\in Y(\overline{K})$ for every $x\in X(\overline{K})$, and we write $f=g|_{X}$. We have $g_{1}|_{X}=g_{2}|_{X}$ if and only $g_{1}(x)=g_{2}(x)$ for all $x\in X(\overline{K})$. For a morphism $f:X\to Y$ and a subvariety $V\subseteq Y$ defined by polynomials $P_i(y_{1},\ldots,y_{n})$, the {\em preimage} $f^{-1}(V)$ is the variety defined by the polynomials $P_{i}(f_{1}(\vec{x}),\ldots,f_{n}(\vec{x}))$ and by the polynomials defining $X$. The image $f(X(\overline{K}))$ need not be the set of points of a variety, though it is a {\em constructible set} (Chevalley; see \cite[\S I.8, Cor.~2 to Thm.~3]{Mumford}), meaning a finite union of intersections $U\cap W$, where $U$ is open and $W$ is closed.

\subsection{Degrees}\label{sss:degree}

Let $V\subseteq\mathbb{A}^{n}$ be a pure-dimensional variety over $K$ with $\dim(V)=d$. The {\em degree} $\deg(V)$ of $V$ is the number of points in the set $(V\cap L)(\overline{K})$, where $L$ is a generic $(n-d)$-dimensional affine subspace of $\mathbb{A}^n$ (by \cite[\S II.3.1.2, Thm.]{DS98} the definition makes sense and $\deg(V)$ is finite).

We can extend the definition of degree to general varieties $V$: $\deg(V)$ is the sum of the degrees of the pure-dimensional parts of $V$.
The bound $\deg(V_{1}\cup V_{2})\leq\deg(V_{1})+\deg(V_{2})$ holds for any $V_{1},V_{2}$ directly by definition. 
If $V$ is the union of irreducible components $V_{i}$, then $\deg(V)=\sum_{i}\deg(V_{i})$.

By a generalization of \textit{B\'ezout's theorem} due to Fulton and Macpherson, as in \cite[Ex.~8.4.6]{Ful84}, \cite[(2.26)]{zbMATH03880868}, or \cite[\S II.3.2.2, Thm.]{DS98}, for $V_{1},V_{2}$ pure-dimensional we have
\begin{equation}\label{eq:bezout}
\deg(V_1\cap V_2)\leq \deg(V_{1})\deg(V_{2}).
\end{equation}
By our definition of degree, \eqref{eq:bezout} holds for $V_1,V_2$ not necessarily pure-dimensional as well.

If $V$ is defined by a single polynomial equation $P=0$ with $\deg(P)>0$, then $\deg(V)\leq\deg(P)$, and equality holds if $P$ has no repeated factors. By B\'ezout, if $V$ is defined by many equations $P_{i}=0$, then $\deg(V)\leq\prod_{i}\deg(P_{i})$.

Following \cite{BDH24}, for a morphism $f:\mathbb{A}^{n}\rightarrow\mathbb{A}^{m}$ given by an $m$-tuple of polynomials $f_{i}$, we define the {\em maximum degree} of $f$ to be $\mdeg(f):=\max_{i}\deg(f_{i})$. For a morphism $f:X\rightarrow Y$, we define $\mdeg(f)$ to be the minimum of $\mdeg(g)$ over all $g:\mathbb{A}^{n}\rightarrow\mathbb{A}^{m}$ with $g|_{X}=f$.

We can bound the degree of images and preimages of varieties.

\begin{lemma}\label{le:zarimdeg}
Let $V\subseteq\mathbb{A}^{n}$ and $W\subseteq\mathbb{A}^{m}$ be varieties and $f:V\rightarrow\mathbb{A}^{m}$ a morphism. Then
\begin{enumerate}[(a)]
\item\label{le:zarimdeg-im} $\deg(\overline{f(V)})\leq\deg(V)\mathrm{mdeg}(f)^{\dim(\overline{f(V)})}$, and
\item\label{le:zarimdeg-pre} $\deg(f^{-1}(W))\leq\deg(V)\deg(W)\mathrm{mdeg}(f)^{\dim(\overline{f(V)})}$.
\end{enumerate}
\end{lemma}

\begin{proof}
See~\cite[Lem.~2.3]{BDH21} and~\cite[Lem.~2.4]{BDH24}.
\end{proof}

An important result of algebraic geometry is that, given an irreducible variety $V$ and a map $f:V\rightarrow\mathbb{A}^{n}$, all fibres $f^{-1}(x)$ have dimension $\geq\dim(V)-\dim(\overline{f(V)})$, with equality holding for a generic fibre. Below is a quantitative version of this statement.

\begin{proposition}\label{pr:degexc}
Let $V$ be an irreducible variety and let $f:V\rightarrow\mathbb{A}^{n}$ be a morphism. Call $u:=\dim(V)-\dim(\overline{f(V)})$. Then for every $x\in f(V)$ we have $\dim(f^{-1}(x))\geq u$. Moreover, there is a proper subvariety $Z\subsetneq\overline{f(V)}$ for which
\begin{equation*}
\deg(Z)\leq\mathrm{mdeg}(f)^{\dim(\overline{f(V)})-1}\deg(V)
\end{equation*}
and for which every $x\in f(V)$ with $\dim(f^{-1}(x))>u$ is contained in $Z$.
\end{proposition}

\begin{proof}
See \cite[\S I.8]{Mumford} and \cite[Prop.~3.2]{BDH24}.
\end{proof}

\subsection{Degree of intersections}

We show now how to bound the degree of intersections of varieties. One can always repeatedly apply B\'ezout, but when the number of intersecting varieties is large (or infinite) the na\"{i}ve bound thus obtained may be unmanageable. The results of this subsection can be seen as explicit versions of \cite[Thm.~1.10, Cor.~1.12]{LP11}. Our technique is already essentially contained in \cite[\S 3.1]{BDH21} and \cite[\S 2.6]{BDH24} (where it was used for a different purpose, namely escape from subvarieties), although we tweak it to make it slightly more general and suited to our needs.

\begin{lemma}\label{le:weighteddeg}
For any $D\geq 1$, define the function $f_{D}$ as follows: for any variety $X\subseteq\mathbb{A}^{n}$, if $\{X_{j}\}_{j}$ is the finite collection of irreducible components of $X$, set
\begin{equation*}
f_{D}(X):=\sum_{j}\deg(X_{j})D^{\dim(X_{j})}.
\end{equation*}
Then, for any two varieties $Y,Z\subseteq\mathbb{A}^{n}$ with $\deg(Z)\leq D$, we have $f_{D}(Y\cap Z)\leq f_{D}(Y)$.
\end{lemma}

\begin{proof}
For any variety $X$, if we partition the collection of its components $\{X_{j}\}_{j}$ into two subsets, say for simplicity $\{X_{j}\}_{j\leq J}$ and $\{X_{j}\}_{j>J}$, and consider their unions $X_{\leq J}$ and $X_{>J}$, we clearly have $f_{D}(X)=f_{D}(X_{\leq J})+f_{D}(X_{>J})$. Thus, it is sufficient to prove the result for $Y$ irreducible.

If $Y=Y\cap Z$ then $f_{D}(Y\cap Z)=f_{D}(Y)$, and if $Y\cap Z=\emptyset$ then $f_{D}(Y\cap Z)=0$. In both cases we are done, so assume otherwise. We must have $0\leq\dim(Y\cap Z)\leq\dim(Y)-1$. If $\{X_{j}\}_{j}$ is the collection of irreducible components of $Y\cap Z$, by B\'ezout
\begin{align*}
f_{D}(Y\cap Z) & =\sum_{j}\deg(X_{j})D^{\dim(X_{j})}\leq D^{\dim(Y)-1}\sum_{j}\deg(X_{j})=D^{\dim(Y)-1}\deg(Y\cap Z) \\
 & \leq D^{\dim(Y)-1}\deg(Y)\deg(Z)\leq D^{\dim(Y)}\deg(Y)=f_{D}(Y),
\end{align*}
proving the result.
\end{proof}

\begin{corollary}\label{co:inters}
Let $\{Z_{i}\}_{i\in I}$ be a (not necessarily finite) collection of varieties inside $\mathbb{A}^{n}$, with $\dim(Z_{i})\leq d$ and $\deg(Z_{i})\leq D$ for all $i\in I$. Let $Z=\bigcap_{i\in I}Z_{i}$. Then
\begin{enumerate}[(a)]
\item\label{co:inters-int} $Z=\bigcap_{i\in I'}Z_{i}$ for some $I'\subseteq I$ with $|I'|\leq 1+(d+1)D^{d+1}$,
\item\label{co:inters-deg} $\deg(Z)\leq D^{d+1}$, and
\item\label{co:inters-extra} for any $Y\subseteq\mathbb{A}^{n}$ with $\dim(Y)=d'$ and $\deg(Y)=D'$, calling $\hat{d}=\min\{d',\dim(Z_{i})\}$, we have $\deg(Y\cap Z)\leq D'D^{\hat{d}+1}$.
\end{enumerate}
\end{corollary}

\begin{proof}
Since $\mathbb{A}^{n}$ is Noetherian, there is a finite subset of $I$ (say $\{1,2,\ldots,J\}$, after renaming the indices) with $\bigcap_{i\leq J}Z_{i}=Z$. Choose the smallest such subset, and from now on we may assume that $I=\{1,2,\ldots,J\}$. We reorder $I$ as follows. Choose $Z_{1}$ arbitrarily, and assume that we are done ordering up to some $j<J$. If $Z_{(j)}:=\bigcap_{i\leq j}Z_{i}$, there is an irreducible component $X$ of $Z_{(j)}$ not contained in $Z$, otherwise we would contradict the minimality of $I$. Among such components, choose one $X$ having largest dimension, and then choose $Z_{j+1}$ such that $X\subsetneq Z_{j+1}$.

By the minimality of $I$ we must have $Z_{1}=Z_{(1)}\supsetneq Z_{(2)}\supsetneq\ldots\supsetneq Z_{(J)}=Z$. Moreover, by the ordering chosen above, we have the following property: there are indices $1=i_{d+1}\leq i_{d}\leq i_{d-1}\leq\ldots\leq i_{1}\leq i_{0}=J$ such that, if $i_{d'+1}<j\leq i_{d'}$, the number of $d'$-dimensional components in $Z_{(j)}$ is strictly smaller than in $Z_{(j-1)}$, and if $j\geq i_{d'}$ the $d'$-dimensional components in $Z_{(j)}$ are the same as those of $Z$.

Now, $f_{D}(Z_{(j)})\leq f_{D}(Z_{(i)})$ whenever $j\geq i$ by Lemma~\ref{le:weighteddeg}. This allows us to give an upper bound on the number $n(d',j)$ of irreducible components of dimension $d'$ inside $Z_{(j)}$. In fact, since we have the bounds
\begin{equation*}
\sum_{\substack{\text{$X$ irr.comp.of $Y$} \\ \dim(X)=d'}}\deg(X)D^{d'}\leq f_{D}(Y)\leq\deg(Y)D^{\dim(Y)}
\end{equation*}
valid for any $Y$ by definition, and since $\deg(Z_{i})\leq D$ for all $i$, we obtain
\begin{align*}
n(d',j) & =\sum_{\substack{\text{$X$ irr.comp.of $Z_{(j)}$} \\ \dim(X)=d'}}1\leq\sum_{\substack{\text{$X$ irr.comp.of $Z_{(j)}$} \\ \dim(X)=d'}}\deg(X)\leq f_{D}(Z_{(j)})D^{-d'} \\
 & \leq f_{D}(Z_{(1)})D^{-d'}=f_{D}(Z_{1})D^{-d'}\leq D^{d-d'+1}
\end{align*}
for all $d',j$. Moreover, by our ordering of the indices, $n(d',j)<n(d',j-1)$ whenever $i_{d'+1}<j\leq i_{d'}$. Therefore, for every $d'$ we have $i_{d'}-i_{d'+1}\leq n(d',i_{d'+1})\leq D^{d-d'+1}$, and the bound
\begin{equation*}
J=1+\sum_{d'=0}^{d}(i_{d'}-i_{d'+1})\leq 1+\sum_{d'=0}^{d}D^{d-d'+1}\leq 1+(d+1)D^{d+1}
\end{equation*}
proves \eqref{co:inters-int}.

Using again Lemma~\ref{le:weighteddeg},
\begin{equation*}
\deg(Z)=\sum_{\text{$X$ irr.comp.of $Z$}}\deg(X)\leq f_{D}(Z)\leq f_{D}(Z_{1})\leq D^{d+1},
\end{equation*}
proving \eqref{co:inters-deg}. We have $\dim(Y\cap Z_{1})\leq\hat{d}$ and $\deg(Y\cap Z_{1})\leq D'D$, so
\begin{equation*}
\deg(Y\cap Z)=\sum_{\text{$X$ irr.comp.of $Y\cap Z$}}\deg(X)\leq f_{D}(Y\cap Z)\leq f_{D}(Y\cap Z_{1})\leq D'D^{\hat{d}+1},
\end{equation*}
proving \eqref{co:inters-extra}.
\end{proof}

\section{Linear algebraic groups}\label{se:linalg}

In this section we define linear algebraic groups, fix the relative notations, and use the tools of Section~\ref{se:var} to prove some general facts that will be used later. We shall discuss the special case of almost simple groups in more depth in Section~\ref{se:almsim}.

\subsection{Definition and basic properties}\label{se:linalg-basic}

The space of $n\times n$ matrices $\mathrm{Mat}_{n}$ is the affine space $\mathbb{A}^{n^{2}}$ endowed with a {\em multiplication map}, i.e.\ a morphism $\cdot:\mathbb{A}^{n^{2}}\times\mathbb{A}^{n^{2}}\rightarrow\mathbb{A}^{n^{2}}$ defined by the usual matrix multiplication. Clearly, $\mathrm{mdeg}(\cdot)=2$.  The {\em general linear group} $\mathrm{GL}_{n}$ is commonly defined as an open set inside $\mathrm{Mat}_{n}$. We shall however need to work with $\mathrm{GL}_{n}$ as a (Zariski-closed) variety, so we define instead
\begin{equation*}
\mathrm{GL}_{n}:=\left\{ x\in\mathrm{Mat}_{n+1}: \ \ x_{i,n+1}=x_{n+1,i}=0 \ \ (1\leq i\leq n), \ \ \det(x|_{n\times n})\cdot x_{n+1,n+1}=1 \right\},
\end{equation*}
where $x_{i,j}$ is the $(i,j)$-th entry of $x$ and $x|_{n\times n}$ is the restriction to the $n\times n$ upper left corner of $x$. The {\em determinant} is the morphism $\det:\mathrm{GL}_{n}\rightarrow\mathbb{A}^{1}$ given by taking the determinant of the $n\times n$ upper left corner of $x$ (we just write $\det(x)$ for simplicity); it has maximum degree $\mathrm{mdeg}(\det)=n$. The {\em inversion map} is the morphism $^{-1}:\mathrm{GL}_{n}\rightarrow\mathrm{GL}_{n}$ sending $x$ to $y$ with $y|_{n\times n}=\mathrm{adj}(x|_{n\times n})\cdot x_{n+1,n+1}$ and $y_{n+1,n+1}=\det(x|_{n\times n})$, where the {\em adjugate} $\mathrm{adj}(z)$ is defined by $\mathrm{adj}(z)_{i,j}=(-1)^{i+j}M_{j,i}$, with $M_{j,i}$ being the $(j,i)$-th minor of $z$.

We occasionally work with a specified $n$-dimensional $K$-vector space $L$, and denote by $\mathrm{Mat}(L)$ and $\mathrm{GL}(L)$ respectively the space $\mathrm{Mat}_{n}$ and the group $\mathrm{GL}_{n}$ over $K$.

A {\em (linear) algebraic group}\footnote{Since we work exclusively in the affine space, there is no need to distinguish between ``algebraic groups'' and ``linear algebraic groups''. See \cite[\S 8.6]{Hum95b}.} $G$ is a subvariety of $\mathrm{GL}_{n}$ closed under the multiplication and inversion maps \cite[p.~51]{Bor91}, and in this case we write $G\leq\mathrm{GL}_{n}$. The maximum degree of these maps may change when restricted to $G$, and the degree of $G$ and its maps may also differ depending on how we choose to represent $G$. We call $H$ an {\em algebraic subgroup}  (or simply {\em subgroup}) of $G$, and write $H\leq G$, if it is both a subvariety of $G$ and an algebraic group (not necessarily defined over the same field $K$). A {\em normal subgroup} $H\unlhd G$ is a subgroup for which, for every $x\in G(\overline{K})$, the automorphism $\varphi:G\rightarrow G$ defined by $\varphi(g)=xgx^{-1}$ satisfies $\varphi(H)=H$; a {\em characteristic subgroup} $H\csgp G$ is a normal subgroup for which the above holds for all automorphisms $\varphi$, not only the inner ones.

One can define the {\em Lie algebra of $G$} by endowing the tangent space of $G$ at $e$ with a Lie algebra structure: see~\cite[\S 3.5]{Bor91} or \cite[\S 9.1]{Hum95b} for details. We use the conventional fraktur notation, such as $\mathfrak{gl}_{n}$ and $\mathfrak{g}$, to denote Lie algebras. We have $\dim(G)=\dim(\mathfrak{g})$ for all $G$.

We now define several objects inside $G$. The {\em identity component} $G^{\mathrm{o}}$ is the connected component of the identity of $G$. If $G=G^{\mathrm{o}}$, or equivalently if $G$ is connected, then $G$ is irreducible~\cite[Cor.~1.35]{Mil17}. For any element $g\in G(\overline{K})$, any set $\Lambda\subseteq G(\overline{K})$, and any variety $V\subseteq G$, their {\em centralizers} in $G$ are
\begin{align*}
C_{G}(g) & :=\{x\in G:gx=xg\}, & C_{G}(\Lambda) & :=\bigcap_{g\in\Lambda}C_{G}(g), & C_{G}(V) & :=C_{G}(V(\overline{K})).
\end{align*}
Every centralizer is an algebraic subgroup of $G$. The {\em centre} of $G$ is defined as $Z(G):=C_{G}(G)$, and we have $Z(G)\csgp G$.

A {\em torus} $T$ of $G$ is a subgroup of $G$ isomorphic to the product of $m$ copies of $\mathrm{GL}_{1}$ for some $m\geq 1$ \cite[\S 8.5]{Bor91}; $T$ is a {\em maximal torus} if it has maximal $m$ among all tori of $G$. A {\em Cartan subgroup} of $G$ is a subgroup of the form $C_{G}(T)$ for some maximal torus $T$. If $G$ is connected then all maximal tori are conjugate \cite[Cor.~11.3(1)]{Bor91}, and the Cartan subgroups of $G$ have the same dimension, which is called the {\em rank} $\mathrm{rk}(G)$ of $G$. Throughout the rest of the paper, when dealing with an algebraic group $G$ we use
\begin{align}\label{eq:notation}
d & =\dim(G), & D & =\deg(G), & r & =\mathrm{rk}(G), & \iota & =\mathrm{mdeg}(^{-1}:G\rightarrow G).
\end{align}

Let $G\leq\mathrm{GL}_{n}\subseteq\mathrm{Mat}_{n+1}$ be defined over $K$. An element $g\in G(K)$ is {\em unipotent} if $(g|_{n\times n}-\mathrm{Id}_{n})^{m}=0$ for some $m\geq 1$, and it is {\em semisimple} if it is conjugate to some diagonal matrix in $G(\overline{K})$ \cite[\S 4.1]{Bor91}. There are unique elements $g_{s},g_{u}$, respectively semisimple and unipotent, such that $g=g_{s}g_{u}=g_{u}g_{s}$ \cite[p.~81, Cor.~1(1)]{Bor91}. An element $g\in G$ is {\em regular} if $C_{G}(g)$ has the smallest possible dimension, which is $\dim(C_{G}(g))=\mathrm{rk}(G)$. We use the notation
\begin{align*}
G^{\mathrm{rss}} & =\{g\in G:\text{$g$ regular and semisimple}\}, & G^{\mathrm{un}} & =\{g\in G:\text{$g$ unipotent}\}, \\
G^{\mathrm{run}} & =\{g\in G:\text{$g$ regular and unipotent}\}, & G^{\mathrm{irr}} & =\{g\in G:\text{$g$ not regular}\}.
\end{align*}
We adopt the shorthand $V^{\mathrm{rss}}=V\cap G^{\mathrm{rss}}$ for varieties $V\subseteq G$, and $X^{\mathrm{rss}}=X\cap G^{\mathrm{rss}}(\overline{K})$ for subsets $X\subseteq G(\overline{K})$ (and similarly for the other notations).

For $G$ connected, a {\em Borel subgroup} $B$ is a maximal connected solvable subgroup of $G$. Its unipotent part $U=B^{\mathrm{un}}$ is a maximal connected unipotent subgroup of $G$. If $\mathcal{B}$ is the collection of Borel subgroups of $G$, then
\begin{align*}
R(G) & =\left(\bigcap_{B\in\mathcal{B}}B\right)^{\mathrm{o}}, & R_{u}(G) & =R(G)^{\mathrm{un}}
\end{align*}
are respectively the {\em radical} and the {\em unipotent radical} of $G$ \cite[\S 11.21]{Bor91}. The definitions above are invariant under automorphisms of $G$, meaning in particular that $R_{u}(G)\csgp G$. A connected $G$ is called {\em reductive} if $R_{u}(G)=\{e\}$, and {\em semisimple} if $R(G)=\{e\}$. If $G$ is reductive, Cartan subgroups and maximal tori coincide \cite[\S 13.17, Cor.~2(c)]{Bor91}, so $\mathrm{rk}(G)$ is the dimension of any maximal torus. If $G$ is also semisimple, we have further properties on the sets of regular, semisimple, and unipotent elements: $G^{\mathrm{rss}}$ is open and dense \cite[\S 2.5]{Hum95a}, $G^{\mathrm{un}}$ is closed and irreducible of dimension $\dim(G)-\mathrm{rk}(G)$ \cite[\S 4.2]{Hum95a}, $G^{\mathrm{run}}$ is nonempty \cite[\S 4.5]{Hum95a}, and $G^{\mathrm{irr}}$ is closed and proper (since the set of regular elements is open and dense \cite[\S 1.4]{Hum95a}).

The degree of many objects above can be bounded effectively.

\begin{lemma}\label{le:tudeg}
Let $G\leq\mathrm{GL}_{n}$ be a connected algebraic group with $d=\dim(G)$, $D=\deg(G)$, and $\iota=\mathrm{mdeg}(^{-1})$, defined over a field $K$.
\begin{enumerate}[(a)]
\item\label{le:tudeg-t} For any maximal torus $T$ of $G$, $\deg(T)\leq D$.
\item\label{le:tudeg-b} For any Borel subgroup $B$ of $G$, $\deg(B)\leq D$.
\item\label{le:tudeg-unip} For any $H\leq G$ connected unipotent, $\deg(H)\leq\dim(H)^{\dim(H)}$.
\item\label{le:tudeg-u} For any maximal connected unipotent subgroup $U$ of $G$, $\deg(U)\leq\min\{D,d^{d}\}$.
\item\label{le:tudeg-ru} For the unipotent radical $R_{u}(G)$ of $G$, $\deg(R_{u}(G))\leq\min\{(nD)^{d+1},d^{d}\}$.
\end{enumerate}
\end{lemma}

\begin{proof}
Each of the subgroups $T,B,U$ as above is respectively of the form $(G\cap T')^{\mathrm{o}},(G\cap B')^{\mathrm{o}},(G\cap U')^{\mathrm{o}}$ for $T',B',U'$ of the same type in $\mathrm{GL}_{n}$ \cite[Prop.~11.14(2)]{Bor91}. All Borel subgroups $B'$ of $\mathrm{GL}_{n}$ are conjugate to each other \cite[\S 11.1]{Bor91}, and similarly for $T',U'$ \cite[Cor.~11.3]{Bor91}. Thus, it is enough to choose $T',B',U'$ whose degrees we can bound well. Choose respectively the diagonal maximal torus $T'$, the group $B'$ of upper triangular matrices, and the subgroup $U'$ of $B'$ whose elements have all diagonal entries equal to $1$; hence, $T',B',U'$ are intersections of $\mathrm{GL}_{n}$ with varieties of degree $1$. The bound in \eqref{le:tudeg-t} and \eqref{le:tudeg-b} and the first bound in \eqref{le:tudeg-u} all follow from B\'ezout.

By definition, $R_{u}(G)$ is the unipotent part of the radical $R(G)=\left(\bigcap_{B\in\mathcal{B}}B\right)^{\mathrm{o}}$. The unipotent part is defined by equations of degree $\leq n$ since, by what we said about Borel subgroups of $\mathrm{GL}_{n}$, $R(G)$ lies in a conjugate of the set of upper triangular matrices. Furthermore, the Borel subgroups of $G$ have degree bounded by \eqref{le:tudeg-b}. Then, for some varieties $Z_{i}$ of degree $\leq n$, Corollary~\ref{co:inters}\eqref{co:inters-deg}--\eqref{co:inters-extra} yields
\begin{align*}
\deg(R_{u}(G)) & =\deg\left(\left(\bigcap_{B\in\mathcal{B}}B\right)^{\mathrm{o}}\cap\bigcap_{i\in I}Z_{i}\right)\leq\deg\left(\left(\bigcap_{B\in\mathcal{B}}B\right)^{\mathrm{o}}\right)\cdot n^{\dim(R(G))+1} \\
 & \leq D^{\dim(B)+1}n^{\dim(R(G))+1}\leq(nD)^{d+1},
\end{align*}
which gives the first bound in \eqref{le:tudeg-ru}.

Now let $H\leq G$ be connected unipotent. By definition $H$ is also solvable, so there is some $z\in\mathrm{GL}_{n}(\overline{K})$ for which $H'=zHz^{-1}$ is upper triangular by the Lie-Kolchin theorem \cite[Cor.~10.5]{Bor91}. $H'$ is still unipotent of dimension $\dim(H)$, so we can write $H=\overline{U_{1}U_{2}\ldots U_{\dim(H)}}$ for some $1$-dimensional irreducible subgroups $U_{i}$ each generated by a unipotent matrix $u_{i}\in H'$ (see for instance \cite[\S 7.5]{Hum95b}, which is more general and forgoes the Zariski closure at the expense of lengthening the product to $2\dim(H)$ factors). For some $z_{i}\in\mathrm{GL}_{n}(\overline{K})$ we can write $u_{i}=z_{i}v_{i}z_{i}^{-1}$ with $v_{i}$ in Jordan normal form. If $\mathcal{J}_{i}$ is the set of indices $j$ for which $(v_{i})_{j,j+1}$ is nonzero, then the group $V_{i}=z_{i}^{-1}U_{i}z_{i}$ is the variety defined by $x_{j,j}=1$ for all $j$, $x_{j_{1},j_{1}+1}=x_{j_{2},j_{2}+1}$ for all $j_{1},j_{2}\in\mathcal{J}_{i}$, and $x_{j,k}=0$ everywhere else. Therefore $\deg(V_{i})=1$, and using the morphism $f:V_{1}\times\ldots\times V_{\dim(H)}\rightarrow G$ defined by
\begin{equation*}
f(x_{1},\ldots,x_{\dim(H)})=z^{-1}\cdot z_{1}x_{1}z_{1}^{-1}\cdot\ldots\cdot z_{\dim(H)}x_{\dim(H)}z_{\dim(H)}^{-1}\cdot z
\end{equation*}
we conclude that $\overline{f(V_{1}\times\ldots\times V_{\dim(H)})}$ has degree $\leq\dim(H)^{\dim(H)}$ by Lemma~\ref{le:zarimdeg}\eqref{le:zarimdeg-im}. This object is $H$ itself, so we obtain \eqref{le:tudeg-unip}.

Finally, both $U$ and $R_{u}(G)$ are connected unipotent, so \eqref{le:tudeg-unip} implies the second bounds in \eqref{le:tudeg-u} and \eqref{le:tudeg-ru}.
\end{proof}

We can also bound the degree of the centralizer of any set in $G(\overline{K})$. This is an easy application of Corollary~\ref{co:inters}, although a more elementary argument relying on the fact that we intersect linear varieties would also be sufficient.

\begin{corollary}\label{co:centralizer}
Let $G\leq\mathrm{GL}_{n}$ be an algebraic group defined over $K$, with $d=\dim(G)$ and $D=\deg(G)$, and let $\Lambda\subseteq G(\overline{K})$. Then $C_{G}(\Lambda)=C_{G}(\Lambda')$ for some $\Lambda'\subseteq\Lambda$ of size $|\Lambda'|\leq d+1$, and $\deg(C_{G}(\Lambda))\leq D$.
\end{corollary}

\begin{proof}
For any $x\in G(\overline{K})$, the centralizer $C_{G}(x)$ is defined as the set of $g\in G$ with $gx=xg$, which yields a finite number of equations of degree $1$. Thus $C_{G}(x)$ is the intersection of $G$ with varieties $Z_{i,x}$ of degree $1$, for some set of indices $i$. In turn, $C_{G}(\Lambda)$ is the intersection of $C_{G}(\lambda)$ for all $\lambda\in\Lambda$. Apply Corollary~\ref{co:inters}\eqref{co:inters-int} to the collection of $Z_{i,\lambda}$ to obtain the bound on $|\Lambda'|$, and Corollary~\ref{co:inters}\eqref{co:inters-extra} to $G$ and the $Z_{i,\lambda}$ to obtain the bound on $\deg(C_{G}(\Lambda))$.
\end{proof}

\subsection{Escaping from a subgroup}

The tools from \cite{BDH21,BDH24} that we are going to use in Section~\ref{se:dimest} rely on the procedure called {\em escape from subvarieties}, which first appeared in \cite{EMO05}. To produce dimensional estimates for a set $A$ of generators of $G(K)$, we need to be able to say that for any proper subvariety $V\subseteq G$ there is some $g\in A^{k}$ with $g\notin V(\overline{K})$, where $k$ is bounded appropriately in terms of $\deg(V)$.

Here we start with a different object, namely a subgroup $\Gamma$, and a weaker hypothesis, namely that we escape from algebraic subgroups. Thus, we have to prove that escaping from subgroups is enough to escape from every subvariety as well, up to paying a price in degree bounds. The following result shows the contrapositive statement: if $\Gamma$ is large enough and is trapped in a subvariety, then it is trapped in a subgroup as well.

\begin{lemma}\label{le:groupvar}
Let $G$ be a linear algebraic group defined over $K$. Let $\Gamma\leq G(\overline{K})$, and let $V\subsetneq G$ be a proper subvariety. Assume that $\Gamma\subseteq V(\overline{K})$.

Then, either $|\Gamma|\leq\deg(V)^{\dim(V)+1}$, or there is a proper algebraic subgroup $H\lneq G$ with $\Gamma\leq H(\overline{K})$ and $\deg(H)\leq\deg(V)^{\dim(V)+1}$.
\end{lemma}

\begin{proof}
For every proper subvariety $W\subsetneq G$, the stabilizer $\mathrm{Stab}(W)=\{g\in G:Wg=W\}$ is a proper algebraic subgroup of $G$.

We build a sequence of $V_{i}$ with the following properties: $\Gamma\subseteq V_{i}(\overline{K})$, $V_{i}=\bigcap_{\gamma\in S_{i}}V\gamma$ for some $S_{i}\subseteq\Gamma$, and $V_{i}\subsetneq V_{i-1}$. We stop constructing the sequence when either $\Gamma\leq\mathrm{Stab}(V_{i})(\overline{K})$ or $\dim(V_{i})=0$. The starting point is $S_{0}=\{e\}$ and $V_{0}=V$. To construct $V_{i+1}$, assume that we have $V_{i}$ as above, and suppose that there is some $\gamma\in\Gamma\setminus\mathrm{Stab}(V_{i})(\overline{K})$. Then $V_{i}\cap V_{i}\gamma\subsetneq V_{i}$ and $\Gamma=\Gamma\gamma\subseteq V_{i}(\overline{K})\gamma$, so we can choose $S_{i+1}=S_{i}\cup S_{i}\gamma$ and obtain $V_{i+1}=V_{i}\cap V_{i}\gamma$ accordingly. The $V_{i}$ are never empty because they contain $\Gamma$, thus since the affine space is Noetherian we will eventually reach some zero-dimensional $V_{i}$ (unless we stopped because $\Gamma\subseteq\mathrm{Stab}(V_{i})(\overline{K})$).

By the above, we obtained that either $\Gamma\subseteq V_{i}(\overline{K})$ with $\dim(V_{i})=0$, or $\Gamma\leq H(\overline{K})$ for $H=\mathrm{Stab}(V_{i})$; in either case, $V_{i}=\bigcap_{\gamma\in S_{i}}V\gamma$ for some $S_{i}\subseteq\Gamma$. Each $V\gamma$ has the same dimension and degree as $V$, so $\deg(V_{i})\leq\deg(V)^{\dim(V)+1}$ by Corollary~\ref{co:inters}\eqref{co:inters-deg}; if $\Gamma\subseteq V_{i}(\overline{K})$ and $\dim(V_{i})=0$, then $|\Gamma|\leq\deg(V_{i})$ and we are done. We can rewrite $H$ as
\begin{equation*}
H=\mathrm{Stab}(V_{i})=\bigcap_{w\in V_{i}(\overline{K})}w^{-1}V_{i}=\bigcap_{(w,\gamma)\in V_{i}(\overline{K})\times S_{i}}w^{-1}V\gamma,
\end{equation*}
so again by Corollary~\ref{co:inters}\eqref{co:inters-deg} we get $\deg(H)\leq\deg(V)^{\dim(V)+1}$.
\end{proof}

\subsection{Quotients}\label{se:quot}

Even if a linear algebraic group $G$ and a normal algebraic subgroup $H\unlhd G$ are naturally defined as varieties, one may not always be able to see the quotient $G/H$ as a variety in any obvious way. Below, we explain how to do so, following \cite[\S 6]{Bor91}.

Let $H\unlhd G$ be both defined over $K$. A \textit{(geometric) quotient} of $G$ by $H$ is a pair $(\pi,W)$ made of an affine variety $W$ and a surjective morphism $\pi:G\rightarrow W$, both defined over $K$, satisfying the following universal property: if $\alpha:G\rightarrow Z$ is a morphism constant on $H$-orbits, there is a unique morphism $\beta:W\rightarrow Z$ such that $\alpha=\beta\circ\pi$, and if $\alpha$ is a $K$-morphism of $K$-varieties then so is $\beta$.

By \cite[Thm.~6.8]{Bor91}, under the conditions above on $G$ and $H$, a geometric quotient always exists and is unique, and $W$ is also an algebraic group defined over $K$. We use the notation $G/H$ for $W$, and we say that $G/H$ is the \textit{quotient} of $G$ by $H$ (forgoing $\pi$). In this case, the quotient $G/H$ also coincides with the \textit{categorical quotient} \cite[\S 6.16]{Bor91}, so we ignore the distinction.

By \cite[Prop.~6.4(b)]{Bor91}, $\dim(G/H)=\dim(G)-\dim(H)$. Furthermore, since $G/H$ is an algebraic group, we can represent it as a variety inside $\mathrm{GL}_{m}$ for some possibly large $m$. The result below gives an upper bound for $m$ and $\deg(G/H)$, at the expense of possibly defining $G/H$ over $\overline{K}$ (which simplifies the matter without being a problem for us later).

\begin{proposition}\label{pr:quot}
Let $G\leq\mathrm{GL}_{n}\subseteq\mathrm{Mat}_{n+1}$ be an algebraic group defined over a field $K$, with $d=\dim(G)$ and $D=\deg(G)$. Let $H\unlhd G$ be a normal subgroup, defined over $K$ by polynomials of degree $\leq\Delta$ (excluding the ones defining $G$ itself).

Then $G/H$ is an algebraic group. More precisely, there are algebraic groups $\hat{H}\unlhd\hat{G}\leq\mathrm{GL}_{2n+3}$ and $Q\leq\mathrm{GL}_{m}$ and a morphism $\hat{\beta}:\hat{G}\rightarrow Q$ (possibly over $\overline{K}$) such that
\begin{enumerate}[(a)]
\item\label{pr:quot-iso} there is a morphism of algebraic groups $\hat{G}\rightarrow G$ of maximal degree $1$ with rational inverse (so in particular $\hat{G}(\overline{K})\simeq G(\overline{K})$ as abstract groups), and the same holds for its restriction $\hat{H}\rightarrow H$,
\item\label{pr:quot-q} $Q\simeq\hat{G}/\hat{H}$, meaning that $Q$ satisfies the aforementioned universal property (possibly over $\overline{K}$), and
\item\label{pr:quot-deg} the following quantitative bounds hold:
\begin{align*}
M & =\binom{n^{2}+\Delta}{\Delta}\leq(n^{2}+\Delta)^{\min\{n^{2},\Delta\}}, & m & \leq 2^{2M}, \\
\deg(\hat{G}) & \leq M2^{M+n^{2}+4}\Delta D, & \deg(\hat{H}) & \leq M2^{M+n^{2}+4}\Delta\deg(H), \\
\deg(Q) & \leq M^{d+1}2^{M+n^{2}+d+5}\Delta^{d+1}D, & \mathrm{mdeg}(\hat{\beta}) & \leq 2M\Delta.
\end{align*}
\end{enumerate}
\end{proposition}

\begin{proof}
We follow the path laid out in \cite{Bor91}, pasting together parts of the proofs of \cite[\S 1.9 and Thms.~5.1--5.6--6.8]{Bor91}. We work everywhere over $\overline{K}$ for simplicity.

By our definition of $\mathrm{GL}_{n}$ in Section~\ref{se:linalg-basic}, we are already keeping track of $\det(g)^{-1}$ in the definition of $g\in G$. We will now use a new $\hat{G}$ to keep track of one more determinant (to be defined in the future through some function $F$) and of the entire $g^{-1}$. In brief, using $\oplus$ to denote the diagonal join of matrices, if an element of $\mathrm{GL}_{n}\subseteq\mathrm{Mat}_{n+1}$ is of the form $a\oplus\det(a)^{-1}$, we pass to $a\oplus\det(a)^{-1}\oplus a^{-1}\oplus\det(a)\oplus y^{-1}\oplus y$ inside $\mathrm{GL}_{2n+3}\subseteq\mathrm{Mat}_{2n+4}$, with $y$ defined using $F$.

Let $R_{1}:=\overline{K}[x_{11},x_{12},\ldots,x_{n+1,n+1}]$ and $R_{2}:=\overline{K}[x_{11},x_{12},\ldots,x_{2n+2,2n+2}]$, and fix $F\in R_{2}$ to be chosen later. Call $\pi_{ij}$ the restriction map sending $x\in\mathrm{Mat}_{2n+4}$ to the square submatrix whose corners are the $(i,i)$-th, $(i,j)$-th, $(j,i)$-th, and $(j,j)$-th entries (with $i\leq j$). Define
\begin{align*}
\mathcal{O}:=\ & ([1,n+1]\times[n+2,2n+2])\cup([n+2,2n+2]\times[1,n+1]) \\
 & \cup([n+2,2n+1]\times\{2n+2\})\cup(\{2n+2\}\times[n+2,2n+1]) \\
 & \cup([1,2n+2]\times\{2n+3\})\cup(\{2n+3\}\times[1,2n+2]) \\
 & \cup([1,2n+3]\times\{2n+4\})\cup(\{2n+4\}\times[1,2n+3]), \\
X:=\ & \left\{ x\in\mathrm{Mat}_{2n+4}: \begin{array}{l} x_{ij}=0 \ \ ((i,j)\in\mathcal{O}), \\ \pi_{1,n}(x)\cdot\pi_{n+2,2n+1}(x)=\mathrm{Id}_{n}, \\ x_{n+1,n+1}\cdot x_{2n+2,2n+2}=1, \\ F(\pi_{1,2n+2}(x))\cdot x_{2n+3,2n+3}=1, \\ x_{2n+3,2n+3}\cdot x_{2n+4,2n+4}=1 \end{array} \right\}.
\end{align*}
The variety $\hat{G}:=X\cap\pi_{1,n+1}^{-1}(G)$ is an algebraic group $\hat{G}\leq\mathrm{GL}_{2n+3}\subseteq\mathrm{Mat}_{2n+4}$, since by construction $\det(\pi_{1,2n+3}(x))\cdot x_{2n+4,2n+4}=1$. If $F(g\oplus g^{-1})\neq 0$ for all $g\in G$, then there is a rational map $G\rightarrow\hat{G}$ that is the inverse of $\pi_{1,n+1}|_{\hat{G}}$, yielding a group isomorphism $\hat{G}(\overline{K})\simeq G(\overline{K})$ and proving \eqref{pr:quot-iso}. Let $f_{1},\ldots,f_{k}\in R_{1}$ be the polynomials defining $H$ in $G$, i.e.\
\begin{equation*}
H=\{g=(g_{ij})_{i,j\leq n+1}\in G:\forall k'\leq k\left(f_{k'}(g_{11},g_{12},\ldots,g_{n+1,n+1})=0\right)\}.
\end{equation*}
By hypothesis $\deg(f_{i})\leq\Delta$, and by construction the same polynomials define $\hat{H}:=X\cap\pi_{1,n+1}^{-1}(H)$ in $\hat{G}$. The group isomorphism $\hat{H}(\overline{K})\simeq H(\overline{K})$ follows from restricting $\pi_{1,n+1}|_{\hat{G}}$ and its inverse rational map. By B\'ezout and Lemma~\ref{le:zarimdeg}\eqref{le:zarimdeg-pre}, we get
\begin{align*}
\deg(\hat{G}) & \leq\deg(G)\cdot 2^{n^{2}+2}(\deg(F)+1), & \deg(\hat{H}) & \leq\deg(H)\cdot 2^{n^{2}+2}(\deg(F)+1).
\end{align*}
We shall pass to $\hat{G},\hat{H}$ at the end of our procedure, whereas for now we use the original $G,H$ for our constructions.

Let $L$ be the $\overline{K}$-vector space spanned by the set $\mathcal{L}$ of all monomials of degree $\leq\Delta$ in the variables $x_{ij}$ ($i,j\leq n+1$); then $\dim(L)=|\mathcal{L}|=M$ with $M$ as in the statement. Each $f_{k}$ is an element of $L$. For every $g\in G(\overline{K})$, let $\rho_{g}:R_{1}\rightarrow R_{1}$ be the right translation by $g$, namely $\rho_{g}(f)(x)=f(xg)$. Then $\rho_{e}$ is the identity map, and $\rho_{g_{1}g_{2}}=\rho_{g_{2}}\circ\rho_{g_{1}}$. The latter property implies that the elements $\rho_{g}(f_{k'})$ for all $k'\leq k$ and all $g\in G(\overline{K})$ span a $\overline{K}$-vector subspace $L'\leq L$ that is invariant under the transformations $\rho_{g}$. Fix a basis $\{\ell_{i}\}_{i\leq M}$ of $L$ whose first $\dim(L')$ members form a basis of $L'$. Every $\ell_{i}$ is a polynomial in $R_{1}$ of degree $\leq\Delta$, and by construction there are polynomials $\tilde{f}_{i,j}\in R_{1}$ of degree $\leq\Delta$ for which $\rho_{g}(\ell_{i})=\sum_{j\leq M}\tilde{f}_{i,j}(g)\ell_{j}$.

Therefore, putting together the facts above, we obtain the following: $\rho_{g}$ is a linear transformation of $L$, i.e.\ $\rho_{g}\in\mathrm{Mat}(L)$, the resulting morphism $\rho^{(L)}:G\rightarrow\mathrm{Mat}(L)$ given by $\rho^{(L)}(g)=\rho_{g}$ has $\mathrm{mdeg}(\rho^{(L)})\leq\Delta$, and $\rho^{(L)}$ naturally restricts to $\rho:G\rightarrow\mathrm{Mat}(L')$ by taking the upper left $(\dim(L')\times\dim(L'))$-corner of the matrix, thus giving again $\mathrm{mdeg}(\rho)\leq\Delta$.

If $I\subseteq\overline{K}[G]$ is the ideal of functions vanishing on $H$, the set $W=L'\cap I$ generates $I$ since $L'$ contains all the $f_{k'}$. We have $\dim(W)\leq\dim(L')\leq M$. The proof of \cite[Thm.~5.1]{Bor91} shows that $H=\{g\in G:\rho_{g}(W)=W\}$.

Next, let $E=\bigwedge^{\dim(W)}(L')$. Then $\rho$ induces a map $\alpha:G\rightarrow\mathrm{Mat}(E)$, with $\mathrm{mdeg}(\alpha)\leq\dim(W)\mathrm{mdeg}(\rho)\leq M\Delta$. Furthermore, \cite[Thm.~5.1]{Bor91} shows that there is a $1$-dimensional subspace $E'\subseteq E$ such that $H=\{g\in G:\alpha(g)(E')=E'\}$. Finally, \cite[Thm.~5.6]{Bor91} shows that there is some subspace $V\subseteq\mathfrak{gl}(E)$ such that the rational map $\beta:G\rightarrow\mathrm{Mat}(V)$ given by $\beta(g)(v)=\alpha(g)v\alpha(g)^{-1}$ has the property that $H=\mathrm{Ker}(\beta)$; \cite[Thm.~6.8]{Bor91} then shows that $\beta(G)$ is a closed subgroup of $\mathrm{Mat}(V)$ with $\beta(G)\simeq G/H$. We have $\dim(E)=\binom{\dim(L')}{\dim(W)}<2^{\dim(L')}\leq 2^{M}$ and $\dim(V)\leq\dim(E)^{2}<2^{2M}$.

Our only problems with the construction of $\beta$ are that $\beta$ is a rational map and not a morphism, and that it goes to $\mathrm{Mat}_{\dim(V)}$ rather than to $\mathrm{GL}_{\dim(V)}$. To be clear on the second issue, since $(\rho_{g})^{-1}=\rho_{g^{-1}}$, the maps $\rho^{(L)},\rho,\alpha,\beta$ do indeed send elements of $G$ to invertible matrices, but we are missing the inverse of the determinant in the lower right corner (needed in our definition of $\mathrm{GL}_{n}$ from Section~\ref{se:linalg-basic}). These two problems are why $\hat{G}$ comes into play: we build a new $\hat{\beta}$ that has the same behaviour as $\beta$, but is also a morphism to $\mathrm{GL}_{\dim(V)}$ because the inverses of $g$ and of $\det(\beta(g))$ can be taken directly from the extra variables of $\hat{G}$.

By construction, every element $\hat{g}\in\hat{G}$ is of the form $\hat{g}=b\oplus y_{1}\oplus y_{2}$ with $y_{2}=F(b)=y_{1}^{-1}$, $b=g\oplus g^{-1}$, and $g\in G\leq\mathrm{GL}_{n}$ (so that in turn $g=a\oplus\det(a)^{-1}$ for some invertible $a\in\mathrm{Mat}_{n}$). Define the morphism 
\begin{align*}
\gamma & :\pi_{1,2n+2}(\hat{G})\rightarrow\mathrm{Mat}_{\dim(V)}, & \gamma(\pi_{1,2n+2}(\hat{g}))(v) & =\alpha(\pi_{1,n+1}(\hat{g}))\cdot v\cdot\alpha(\pi_{n+2,2n+2}(\hat{g})).
\end{align*}
Since $\rho_{g^{-1}}=(\rho_{g})^{-1}$ we have $\rho(g^{-1})=\rho(g)^{-1}$ and $\alpha(g^{-1})=\alpha(g)^{-1}$. Therefore $\alpha(\pi_{n+2,2n+2}(\hat{g}))=\alpha(\pi_{1,n+1}(\hat{g}))^{-1}$ for all $\hat{g}\in\hat{G}$.

Choose $F\in R_{2}$ to be the polynomial $F(x)=\det(\gamma(x))$. By construction, for all $\hat{g}\in\hat{G}$ we have $F(g\oplus g^{-1})=\det(\gamma(g\oplus g^{-1}))=\det(\beta(g))$. In particular $F(g\oplus g^{-1})\neq 0$ for all $g\in G$ (because $\beta$ sends $g$ to an invertible matrix), so by what we said before we obtain \eqref{pr:quot-iso}. Finally, define the morphism $\hat{\beta}:\hat{G}\rightarrow\mathrm{GL}_{\dim(V)}\subseteq\mathrm{Mat}_{\dim(V)+1}$ by $\hat{\beta}(\hat{g})=\gamma(\pi_{1,2n+2}(\hat{g}))\oplus y_{1}$. As $\gamma(\pi_{1,2n+2}(\hat{g}))=\beta(g)$ and $y_{1}=\det(\beta(g))^{-1}$, the construction of $\hat{\beta}$ coincides with that of $\beta$ with the addition of the extra entry for the inverse of the determinant. Hence we have again a closed subgroup $Q:=\hat{\beta}(\hat{G})\simeq\hat{G}/\hat{H}$ of $\mathrm{GL}_{\dim(V)}$, but by passing through $\gamma$ we are now only working with polynomials, and we have \eqref{pr:quot-q}.

To obtain \eqref{pr:quot-deg}, it remains to compute degrees. We have
\begin{align*}
\deg(F) & \leq 2\mathrm{mdeg}(\alpha)\dim(E)\leq M2^{M+1}\Delta, \\
\deg(\hat{G}) & \leq\deg(G)\cdot 2^{n^{2}+2}(\deg(F)+1)\leq M2^{M+n^{2}+4}\Delta D, \\
\deg(\hat{H}) & \leq\deg(H)\cdot 2^{n^{2}+2}(\deg(F)+1)\leq M2^{M+n^{2}+4}\Delta\deg(H), \\
\mathrm{mdeg}(\hat{\beta}) & \leq\mathrm{mdeg}(\gamma)\leq 2\mathrm{mdeg}(\alpha)\leq 2M\Delta, \\
\deg(Q) & \leq\deg(\hat{G})\mathrm{mdeg}(\hat{\beta})^{\dim(G)}\leq M^{d+1}2^{M+n^{2}+d+5}\Delta^{d+1}D,
\end{align*}
where the last line used Lemma~\ref{le:zarimdeg}\eqref{le:zarimdeg-im}.
\end{proof}

\section{Almost simple groups}\label{se:almsim}

In this section we restrict our focus to a special class of algebraic groups, the almost simple groups, with particular emphasis on the adjoint ones. We recall their classification and, for any such $G$, describe several related objects that will be fundamental in our main proof: the Weyl group $\mathcal{W}$ of $G$, the Steinberg endomorphisms $F:G(\overline{\mathbb{F}_{p}})\rightarrow G(\overline{\mathbb{F}_{p}})$ in positive characteristic, and a certain representation $\rho$ of $G$ taken from \cite{Pin98}.

\subsection{Almost simple and adjoint groups}\label{se:adjoint}

A connected linear algebraic group $G$ is {\em almost simple}\footnote{There are many different names for these objects. Some authors call these groups {\em simple}, as in \cite[\S 1.11]{Car93}, \cite[p.~168]{Hum95b}, and \cite[p.~xiv]{MT11}, but they are not necessarily simple as abstract groups. Some call them {\em quasi-simple}, as in \cite[Ex.~8.1.12(4a)]{Spr08}, keeping closer to the convention of abstract groups. Authors that put emphasis on the field of definition use {\em geometrically almost simple}, as in \cite[Def.~19.7]{Mil17}, or {\em absolutely almost simple}, as in \cite[\S 5]{BGT11}, and they may drop the ``almost'', as in \cite{Gar10}. Our work is mostly on the algebraic closure $\overline{K}$, and we deal with finite simple groups as well, so we choose to adopt the term {\em almost simple} as in \cite[Prop.~14.10(3)]{Bor91}.} if it is non-abelian and has no connected normal linear algebraic subgroups except $\{e\}$ and $G$. If $G$ is almost simple then it is also semisimple. If not, we would have $R(G)=G$, and in particular $G$ would be connected solvable, so that by \cite[Thm.~10.6(1)]{Bor91} either $G^{\mathrm{un}}=\{e\}$ or $G^{\mathrm{un}}=G$. Then, \cite[Thm.~10.6(2)]{Bor91} implies in the first case that $G$ is a torus, and in the second case that $\dim(G)\leq 1$, contradicting in both cases the fact that $G$ is not abelian (in the latter case by \cite[\S 20]{Hum95b}).

To every connected semisimple algebraic group $G$ one can associate a {\em Dynkin diagram}, a finite multigraph that is connected if and only if $G$ is almost simple. The connected Dynkin diagrams are completely classified, which allows us to characterize the corresponding groups by types: the possibilities are the {\em classical types} $A_{r\geq 1},B_{r\geq 2},C_{r\geq 3},D_{r\geq 4}$ (where $r$ is the same as the rank of the corresponding group), and the {\em exceptional types} $E_{6},E_{7},E_{8},F_{4},G_{2}$. For $G$ connected, an {\em isogeny} is a surjective morphism $\varphi:G\rightarrow H$ with finite kernel. Every connected almost simple algebraic group is uniquely determined by its Dynkin diagram and isogeny type.

See \cite[\S 1.11]{Car93}, \cite[\S 32 and App.]{Hum95b}, \cite[\S 9]{MT11}, or \cite[\S 24]{Mil17} for more details about Dynkin diagrams, isogenies, almost simple groups, and their classification. Here we only spend a few words on roots and weights. The identification between groups and diagrams passes through a {\em root system} $\Phi$, defined for groups in \cite[\S 16.4]{Hum95b} and for diagrams as in \cite[App.]{Hum95b} and put in relation to each other in \cite[\S 27]{Hum95b}. The {\em roots} are (finitely many) elements spanning a vector space that can be identified with $\mathbb{R}\otimes_{\mathbb{Z}}X(T)$, with $X(T)$ the character group of a maximal torus $T$, and the {\em weights} are the elements of the vector space whose inner products with the roots are all integers. Roots and weights span two lattices, $\Lambda_{r}$ and $\Lambda_{w}$, with $\Lambda_{r}\leq X(T)\leq\Lambda_{w}$ and with $[\Lambda_{w}:\Lambda_{r}]$ finite. The groups $\Lambda_{w}/\Lambda_{r}$ and $\Lambda_{w}/X(T)$ are called the {\em fundamental groups} of $\Phi$ and $G$, respectively (see \cite[\S 31.1 and App.]{Hum95b}). For a fixed Dynkin diagram, isogeny types and fundamental groups for $G$ are in 1-to-1 correspondence (here we really need equality of fundamental groups, not just isomorphism: see \cite[\S 32.1, Thm.]{Hum95b}). The group $\Lambda_{w}/\Lambda_{r}$ is given in \cite[App.~(A.9)]{Hum95b} or \cite[Table~9.2]{MT11}; we shall only need the uniform cruder bound
\begin{equation}\label{eq:indexcartan}
|\Lambda_{w}/\Lambda_{r}|\leq r+1,
\end{equation}
where $r$ is the rank of $G$.

Among the (finitely many) connected almost simple groups $G$ with the same Dynkin diagram, i.e.\ isogenous to each other, there are two extremes: the {\em simply connected} group with $X(T)=\Lambda_{w}$ and the {\em adjoint} group with $X(T)=\Lambda_{r}$ \cite[Def.~9.14]{MT11}. For every $G$ of a given type, there are isogenies from the corresponding simply connected group to $G$ and from $G$ to the adjoint group \cite[Prop.~9.15]{MT11}. As in \cite{LP11}, we work almost exclusively with the adjoint groups.

Another definition independent of the concepts above is the following: a semisimple group $G$ is \textit{adjoint} if and only if its centre $Z(G)$ is trivial \cite[\S 17.g]{Mil17}. For $G$ reductive the {\em adjoint representation} of $G$ is the morphism $\mathrm{Ad}_{G}$ defined in \cite[\S 10.d]{Mil17} as follows: 
\begin{align*}
\mathrm{Ad}_{G} & :G\rightarrow\mathrm{GL}(\mathfrak{g}), & g & \mapsto\mathrm{Ad}_{G,g}, & \text{ with \ } \mathrm{Ad}_{G,g} & :\mathfrak{g}\rightarrow\mathfrak{g}, & x & \mapsto gxg^{-1}.
\end{align*}
We abbreviate $\mathrm{Ad}_{G}(G)$ as $G^{\mathrm{ad}}$; it is isomorphic to $G/Z(G)$, which is an adjoint group by \cite[Cor.~17.62(e)]{Mil17}. If $G\leq\mathrm{GL}_{n}$ has $\mathrm{mdeg}(^{-1})=\iota$, it is clear that
\begin{equation}\label{eq:mdegad}
\mathrm{mdeg}(\mathrm{Ad}_{G})\leq\iota+1\leq n+1.
\end{equation}
In fact, in concrete terms, the action of $G$ on $\mathfrak{g}$ is given by conjugation in $\mathrm{Mat}_{n+1}$, so it has maximal degree $\leq\iota+1$; up to a change of basis (not affecting $\mathrm{mdeg}$), $\mathrm{Mat}_{n+1}$ is the sum of the subspace $\mathfrak{g}$ and a complement of it, and the restriction of the conjugation map to $\mathfrak{g}$ has degree $\leq\iota+1$ too. The bound $\iota\leq n$ comes from using the adjugate map to define inverses.

From any $G$ connected we can naturally descend to an adjoint quotient group. Taking the quotient here is essentially what gives \eqref{th:main-ab} and \eqref{th:main-p} in the main theorem.

\begin{lemma}\label{le:rzdeg}
Let $G\leq\mathrm{GL}_{n}$ be a connected algebraic group with $d=\dim(G)$, $D=\deg(G)$, and $\iota=\mathrm{mdeg}(^{-1})$, defined over a field $K$.

There is some characteristic subgroup $Y\csgp G$, defined by the polynomials defining $G$ and by some additional polynomials over $\overline{K}$ of degree $\leq(\iota+1)d^{d}$, such that
\begin{equation*}
G/Y\simeq(G/R_{u}(G))^{\mathrm{ad}}\simeq(G/R_{u}(G))/Z(G/R_{u}(G)).
\end{equation*}
Moreover there are algebraic groups $\hat{G},\hat{Y}$, there is a morphism
\begin{align*}
\lambda & :\hat{G}\rightarrow G, & \lambda|_{\hat{Y}} & :\hat{Y}\rightarrow Y, & \mathrm{mdeg}(\lambda) & =\mathrm{mdeg}(\lambda|_{\hat{Y}})=1,
\end{align*}
having rational inverse, so that $G(\overline{K})\simeq\hat{G}(\overline{K})$ and $Y(\overline{K})\simeq\hat{Y}(\overline{K})$, and there is a morphism $\hat{\beta}:\hat{G}\rightarrow\hat{G}/\hat{Y}\leq\mathrm{GL}_{m}$ satisfying
\begin{align*}
\mathrm{mdeg}(\hat{\beta}) & \leq 2(n^{2}+(\iota+1)d^{d})^{n^{2}}(\iota+1)d^{d}, & m & \leq 2^{2(n^{2}+(\iota+1)d^{d})^{n^{2}}}.
\end{align*}
\end{lemma}

\begin{proof}
The group $Y$ can be defined as the set of $y\in G$ such that $x^{-1}y^{-1}xy\in R_{u}(G)$ for all $x\in G$. Unipotent radicals and centres are preserved by automorphisms, so $Y$ is characteristic. For any fixed $x\in G$, let $f_{x}:G\rightarrow G$ be defined by $f_{x}(y)=x^{-1}y^{-1}xy$, so that $\mathrm{mdeg}(f_{x})\leq\iota+1$. By Lemma~\ref{le:tudeg}\eqref{le:tudeg-ru} and \cite[Prop.~3]{Hei83}, $R_{u}(G)$ is defined by polynomials (over the algebraic closure $\overline{K}$) of degree $\leq d^{d}$. Thus, $Y$ is defined by the polynomials of $G$ and by polynomials of degree $\leq(\iota+1)d^{d}$.

Applying Proposition~\ref{pr:quot}, we obtain the various assertions on $\hat{G},\hat{Y},\lambda,\hat{\beta},m$.
\end{proof}

\subsection{Classification of almost simple adjoint groups}

The almost simple adjoint groups can all be defined very explicitly by taking quotients, identity components, and preimages of subgroups in appropriate matrix spaces. See \cite[Table~9.2]{MT11} for a list of groups for each type, where the simply connected and adjoint extremes are highlighted. The adjoint groups for each classical type are: $\mathrm{PGL}_{r+1}$ for type $A_{r}$, $\mathrm{SO}_{2r+1}$ for type $B_{r}$, $\mathrm{PCSp}_{2r}$ for type $C_{r}$, and $\mathrm{P}((\mathrm{CO}_{2r}^{+})^{\mathrm{o}})$ for type $D_{r}$. As for the exceptional types, the groups $E_{6}^{\mathrm{ad}},E_{7}^{\mathrm{ad}},E_{8}^{\mathrm{ad}},F_{4}^{\mathrm{ad}},G_{2}^{\mathrm{ad}}$ can be obtained via the adjoint representation on their own Lie algebras; note that for types $E_{8},F_{4},G_{2}$ there is a unique almost simple group of that type.

One can use the definition itself of each adjoint group $G$ above and combine it with Lemmas~\ref{le:zarimdeg}\eqref{le:zarimdeg-im}--\ref{le:rzdeg} to find a concrete way to write $G\leq\mathrm{GL}_{m}$ with bounds on $m$ and $\deg(G)$. However, every adjoint group $G$ is also the image of the adjoint representation $\mathrm{Ad}_{\tilde{G}}$ for any $\tilde{G}$ almost simple with the same Dynkin diagram as $G$. Thus, we can choose some suitable $\tilde{G}$ and define the adjoint groups as subgroups of matrices with more convenient bounds.

\begin{proposition}\label{pr:adjbounds}
Every connected almost simple adjoint algebraic group $G$ of dimension $d$ and rank $r$ can be written as a linear algebraic group $G\leq\mathrm{GL}_{d}$ with $\deg(G)\leq(2r)^{2^{16}r^{2}}$.
\end{proposition}

\begin{proof}
Every connected almost simple adjoint $G$ is the image of $\mathrm{Ad}_{\tilde{G}}:\tilde{G}\rightarrow\mathrm{GL}(\mathfrak{\tilde{g}})$ for any $\tilde{G}$ with the same Dynkin diagram as $G$. We only need to choose a suitable $\tilde{G}$.

We start with the classical types $A_{r},B_{r},C_{r},D_{r}$. Going through the possibilities in \cite[\S 1.11]{Car93} or \cite[Table~9.2]{MT11} for each type, we may choose the following ($\Omega_{C},\Omega_{D}$ below are fixed constant matrices).
\begin{align*}
& \text{Type $A_{r}$:} & \tilde{G} & =\mathrm{SL}_{r+1}=\{x\in\mathrm{Mat}_{r+1}:\det(x)=1\}. \\
& \text{Type $B_{r}$:} & \tilde{G} & =G=\mathrm{SO}_{2r+1}. \\
& \text{Type $C_{r}$:} & \tilde{G} & =\mathrm{Sp}_{2r}=\{x\in\mathrm{Mat}_{2r}:x^{\top}\Omega_{C}x=\Omega_{C}\}. \\
& \text{Type $D_{r}$:} & \tilde{G} & =\mathrm{SO}_{2r}^{+}=\{x\in\mathrm{Mat}_{2r}:\det(x)=1,x^{\top}\Omega_{D}x=\Omega_{D}\}.
\end{align*}
A uniform degree bound for every $\tilde{G}$ as above is $\deg(\tilde{G})\leq(2r+1)2^{(2r+1)^{2}}$ (slightly better bounds appear in \cite[Table~1]{BDH21}). By Lemma~\ref{le:zarimdeg}\eqref{le:zarimdeg-im}, \eqref{eq:mdegad}, and the bound $\dim(G)\leq 2r^{2}+r$, we conclude that
\begin{equation*}
\deg(G)\leq\deg(\tilde{G})\mathrm{mdeg}(\mathrm{Ad}_{\tilde{G}})^{\dim(G)}\leq(2r+1)2^{(2r+1)^{2}}(2r+2)^{2r^{2}+r}<2^{17r^{2}}r^{2r^{2}+r+1}
\end{equation*}
for $G$ of classical type.

We pass now to the exceptional types $E_{6},E_{7},E_{8},F_{4},G_{2}$. We may act as above and use any $\tilde{G}$, for which $\deg(\tilde{G})$ is bounded by some absolute constant, but it is possible to describe $G$ very explicitly. Every $G$ adjoint is the automorphism group of its own finite-dimensional Lie algebra, and in some cases of other algebras too: as already observed in \cite[\S 6.1]{BDH24}, this means that we can define $G$ via quadratic equations
\begin{equation*}
\sum_{i,j}\lambda_{ija}g_{bi}g_{cj}-\sum_{k}\lambda_{bck}g_{ka}=0
\end{equation*}
for all triples $(a,b,c)$, where we have specified a basis $\{e_{i}\}_{i}$ for the non-associative algebra with $g(e_{i})=\sum_{j}g_{ij}e_{j}$ and $e_{i}\circ e_{j}=\sum_{k}\lambda_{ijk}e_{k}$. Hence, we conclude that $\deg(G)\leq 2^{\dim(G)^{3}}$, which for the exceptional types may be bounded uniformly by $(2r)^{2^{16}r^{2}}$, say. As this bound works for the classical types too, we obtain the result.
\end{proof}

\begin{remark}\label{re:a1b2g2}
A few observations about the result above.
\begin{enumerate}[(a)]
\item To bound $\deg(G)$, one could write $G$ directly as a quotient and use Proposition~\ref{pr:quot}. However, this route would give $G\leq\mathrm{Mat}_{m}$ with $m\leq C^{r^{2}}$ and $\deg(G)\leq C^{r^{4}}$ for some absolute constant $C$. Passing afterwards through the adjoint representation would give $G\leq\mathrm{GL}_{d}$, but $\deg(G)$ would grow even more.
\item\label{re:a1b2g2-yes} In order to deal with a specific computation later, i.e.\ \eqref{eq:ha1b2g2}, we now bound $\deg(G)$ more tightly for $G$ connected almost simple adjoint of type $A_{1},B_{2},G_{2}$. For $G=\mathrm{PGL}_{2}$ and $\tilde{G}=\mathrm{SL}_{2}$, we use $\deg(G)\leq\deg(\tilde{G})\mathrm{mdeg}(\mathrm{Ad}_{\tilde{G}})^{\dim(G)}$ and obtain $\deg(G)\leq 16$. For $G=\tilde{G}=\mathrm{SO}_{5}$, \cite[Thm.~1.1]{BBBKR17} yields $\deg(G)\leq 384$. For $G$ of type $G_{2}$, the proof of Proposition~\ref{pr:adjbounds} already shows the inequality $\deg(G)\leq 2^{7^{3}}$.
\item We give references for several algebras that can be used for the groups $G$ of exceptional type. The Lie algebras can be constructed rather explicitly from \cite[\S 3]{HRT01}; even more explicit resources are the multiplication table of $\mathfrak{g}_{2}$ from \cite[\S 22]{FH04}, and the complete bases of all five types from \cite[App.~A.1]{Don23} (based on the aforementioned references). Moreover, as pointed out before, for types $G_{2},F_{4},E_{8}$ there is only one connected almost simple $G$ per type; thus, other constructions not using Lie algebras will also give the same algebraic group. For $G_{2}$ we can use the octonion algebra, given in \cite[\S\S 4.3.2--4.3.4--4.4.3]{Wil09} in three bases with different conditions on the characteristic. For $F_{4}$ we can use the Albert algebra, or a quadratic and a cubic form based on its construction, as in \cite[\S 4.8]{Wil09}. For $E_{8}$ there are no smaller representations than the Lie algebra $\mathfrak{e}_{8}$.
\end{enumerate}
\end{remark}

\subsection{Weyl groups}

For $G$ connected and $T$ a maximal torus, the \textit{Weyl group} $\mathcal{W}$ is the quotient of the normalizer of $T$ by the centralizer of $T$ inside $G$ \cite[\S 1.9]{Car93}. The construction is independent from the choice of $T$, so we can refer to $\mathcal{W}$ as the Weyl group of $G$. For $G$ almost simple, $\mathcal{W}$ is a finite group determined by the Dynkin diagram: in fact, it can be equivalently defined using reflections  of the root system \cite[\S 14.7]{Bor91}. Below we present a list of $\mathcal{W}$ and their sizes, taken from \cite[\S 2.8.4, (3.22), (3.31), (3.39)]{Wil09}; the notation on group extensions $A\times B$, $A^{\displaystyle{\cdot}}B$, $A:B$ and on abelian groups $p^{m+n}$ is as in \cite[\S 1.6]{Wil09}, which follows \cite[p.~xx]{CCNPW85}.
\begin{align*}
& \text{Type $A_{r}$:} & \mathcal{W} & \simeq\mathrm{Sym}(r+1), & |\mathcal{W}| & =(r+1)!, \\
& \text{Type $B_{r}$:} & \mathcal{W} & \simeq\mathrm{Sym}(2)\wr\mathrm{Sym}(r), & |\mathcal{W}| & =2^{r}r!, \\
& \text{Type $C_{r}$:} & \mathcal{W} & \simeq\mathrm{Sym}(2)\wr\mathrm{Sym}(r), & |\mathcal{W}| & =2^{r}r!, \\
& \text{Type $D_{r}$:} & \mathcal{W} & \simeq H \text{ with } [\mathrm{Sym}(2)\wr\mathrm{Sym}(r):H]=2, & |\mathcal{W}| & =2^{r-1}r!, \\
& \text{Type $E_{6}$:} & \mathcal{W} & \simeq\mathrm{GO}_{6}^{-}(2)\simeq\mathrm{U}_{4}(2):2, & |\mathcal{W}| & =51840, \\
& \text{Type $E_{7}$:} & \mathcal{W} & \simeq\mathrm{GO}_{7}(2)\times 2\simeq\mathrm{Sp}_{6}(2)\times 2, & |\mathcal{W}| & =2903040, \\
& \text{Type $E_{8}$:} & \mathcal{W} & \simeq 2^{\displaystyle{\cdot}}\mathrm{GO}_{8}^{+}(2)\simeq2^{\displaystyle{\cdot}}\Omega_{8}^{+}(2):2, & |\mathcal{W}| & =696729600, \\
& \text{Type $F_{4}$:} & \mathcal{W} & \simeq\mathrm{GO}_{4}^{+}(3)\simeq 2^{1+4}:(\mathrm{Sym}(3)\times\mathrm{Sym}(3)), & |\mathcal{W}| & =1152, \\
& \text{Type $G_{2}$:} & \mathcal{W} & \simeq\mathrm{Dih}_{6}\simeq\mathrm{Sym}(3)\times 2, & |\mathcal{W}| & =12.
\end{align*}
As a uniform cruder bound for all $|\mathcal{W}|$ at once, we shall use
\begin{equation}\label{eq:weilg}
|\mathcal{W}|\leq(2r)^{r},
\end{equation}
where $r$ is the rank of $G$.

\subsection{Steinberg endomorphisms}\label{se:steinberg}

In this subsection we work over the field $\overline{\mathbb{F}_{p}}$.

Let $G$ be a linear algebraic group $G\leq\mathrm{GL}_{n}$ defined over $\overline{\mathbb{F}_{p}}$, and let $q$ be a power of $p$. The {\em Frobenius map (with respect to $\mathbb{F}_{q}$)} is the morphism $F_{q}:G\rightarrow G$ given by raising each entry of $g$ to the $q$-th power. A {\em Steinberg endomorphism} \cite[Def.~21.3]{MT11} is an automorphism of abstract groups $F:G(\overline{\mathbb{F}_{p}})\rightarrow G(\overline{\mathbb{F}_{p}})$ such that there are $q$ and $m$ for which $F^{m}$ is (the map on $\overline{\mathbb{F}_{p}}$-points induced by) the Frobenius map of $G$ with respect to $\mathbb{F}_{q}$.

This terminology comes from \cite{MT11}, but is not unanimously adopted in the literature: \cite{LP11} uses the locution ``Frobenius map'' to refer to what we call a Steinberg endomorphism, and ``standard Frobenius map'' to refer to what we call a Frobenius map.

A Steinberg endomorphism is not necessarily a morphism of linear algebraic groups. The classification of Steinberg endomorphisms of connected almost simple groups is known: see \cite[Thm.~22.5]{MT11}, which however is restricted to the simply connected case.

For a Steinberg endomorphism $F$ of $G$, we denote by $G^{F}$ the subgroup of $G(\overline{\mathbb{F}_{p}})$ of points fixed by $F$. The group $G^{F}$ is finite whenever $G$ is almost simple \cite[Thm.~21.5]{MT11}. The classification of all possible $G^{F}$ for any connected almost simple $G$ and any Steinberg endomorphism $F$ is given in \cite[\S 1.19]{Car93} and in \cite[\S 22.2]{MT11}.

\subsection{The representation $\rho$ of $G$ adjoint}

Finally, we must study a particular representation of $G$. We follow Pink's notation in \cite{Pin98}.

Let $G$ be connected almost simple adjoint defined over a field $K$, and let $\tilde{G}$ be the corresponding simply connected group. There is a natural isogeny $\pi:\tilde{G}\rightarrow G$, which induces a linear transformation $d\pi:\tilde{\mathfrak{g}}\rightarrow\mathfrak{g}$; the two Lie algebra have the same dimension, and the same is true for the kernel $\mathfrak{z}$ and cokernel $\mathfrak{z}^{*}$ of $d\pi$. We can write $\tilde{\mathfrak{g}}=\mathfrak{z}\oplus\bar{\mathfrak{g}}$ and $\mathfrak{g}=\bar{\mathfrak{g}}\oplus\mathfrak{z}^{*}$, so that $\bar{\mathfrak{g}}$ is the subalgebra whose copies inside $\tilde{\mathfrak{g}}$ and $\mathfrak{g}$ are identifiable with $\tilde{\mathfrak{g}}/\mathfrak{z}$ and $d\pi(\tilde{\mathfrak{g}})$ respectively; when restricted to these two copies of $\bar{\mathfrak{g}}$, the map $d\pi$ is given by some $\bar{\pi}\in\mathrm{GL}(\bar{\mathfrak{g}})$.

As we already showed in \eqref{eq:mdegad}, the adjoint representation $\mathrm{Ad}_{G}:G\rightarrow\mathrm{GL}(\mathfrak{g})$ has $\mathrm{mdeg}(\mathrm{Ad}_{G})\leq\iota+1$. Following \cite[(1.3)--(1.4)]{Pin98}, we can define $\kappa:\mathrm{Hom}(\mathfrak{g},\tilde{\mathfrak{g}})\rightarrow\mathrm{End}(\mathfrak{g})$ by $\kappa(f)=\mathrm{Id}_{\mathfrak{g}}+d\pi\circ f$, and a morphism $\widetilde{\mathrm{Ad}}_{G}:G\rightarrow\mathrm{Hom}(\mathfrak{g},\tilde{\mathfrak{g}})$ so that $\mathrm{Ad}_{G}=\kappa\circ\widetilde{\mathrm{Ad}}_{G}$. Therefore, $\mathrm{Ad}_{G,g}(v)=v+(d\pi\circ\widetilde{\mathrm{Ad}}_{G}(g))(v)$ for every $g\in G$ and $v\in\mathfrak{g}$; knowing $\mathrm{mdeg}(\mathrm{Ad}_{G})$ and using the fact that $d\pi:\mathfrak{z}\oplus\bar{\mathfrak{g}}\rightarrow\bar{\mathfrak{g}}\oplus\mathfrak{z}^{*}$ acts as the zero map on $\mathfrak{z}$ and as $\bar{\pi}$ on $\bar{\mathfrak{g}}$, we obtain $\mathrm{mdeg}(\widetilde{\mathrm{Ad}}_{G})\leq\iota+1$.

Now, following \cite[Prop.~1.10]{Pin98}, let $\hat{\mathfrak{g}}=\mathfrak{z}\oplus\bar{\mathfrak{g}}\oplus\mathfrak{z}^{*}$, let $di:\tilde{\mathfrak{g}}\rightarrow\hat{\mathfrak{g}}$ be the natural inclusion, and let $d\omega:\hat{\mathfrak{g}}\rightarrow\mathfrak{g}$ be the map induced by $d\pi$: namely, $d\omega$ acts as the zero map, as $\bar{\pi}$, and as the identity on $\mathfrak{z},\bar{\mathfrak{g}},\mathfrak{z}^{*}$ respectively. Define the representation $\hat{\rho}:G\rightarrow\mathrm{GL}(\hat{\mathfrak{g}})$ by $\hat{\rho}(g)=\mathrm{Id}_{\hat{\mathfrak{g}}}+di\circ\widetilde{\mathrm{Ad}}_{G}(g)\circ d\omega$. By the discussion above, $\mathrm{mdeg}(\hat{\rho})\leq\iota+1$. Furthermore, $\hat{\rho}$ is defined over $K$, since this is true for all the objects defined so far.

By \cite[Prop.~1.11(b)]{Pin98}, if $G$ does not have a non-standard isogeny then $\bar{\mathfrak{g}}$ is the unique simple $G$-submodule of $\mathfrak{g}$ and the unique simple quotient $G$-module of $\tilde{\mathfrak{g}}$, and the restriction of $\hat{\rho}$ on $\bar{\mathfrak{g}}$ is irreducible and non-constant. By \cite[Prop.~1.11(c)]{Pin98}, if $G$ has a non-standard isogeny then there is a unique simple $G$-submodule $\bar{\mathfrak{g}}_{s}$ of $\mathfrak{g}$ and there is a unique simple quotient $G$-module $\bar{\mathfrak{g}}_{\ell}$ of $\tilde{\mathfrak{g}}$, and the restrictions of $\hat{\rho}$ on them are irreducible, non-constant, and not equivalent to each other; moreover, we can decompose $\hat{\mathfrak{g}}$ so that $\bar{\mathfrak{g}}_{s}$ and $\bar{\mathfrak{g}}_{\ell}$ are transversal direct summands of $\hat{\mathfrak{g}}$ (see the graphs in \cite[Prop.~1.11(c)]{Pin98}). The restrictions of $\hat{\rho}$ to $\bar{\mathfrak{g}},\bar{\mathfrak{g}}_{s},\bar{\mathfrak{g}}_{\ell}$ in the various cases are denoted by $\alpha^{G},\alpha^{G}_{s},\alpha^{G}_{\ell}$. Up to a change of basis, which does not affect $\mathrm{mdeg}(\hat{\rho})$, we see again that the restrictions $\alpha^{G},\alpha^{G}_{s},\alpha^{G}_{\ell}$ have maximal degree $\leq\iota+1$.

It is time to define the representation $\rho$ of $G$, as given in \cite[p.~1142]{LP11}. Let $U$ be the unipotent part of a Borel subgroup $B$ of $G$; up to isomorphism, $U$ is independent from the choice of $B$. By \cite[Prop.~8.3]{Bor91}, the centre $Z(U)$ is of the form
\begin{align}\label{eq:zu}
Z(U)=\begin{cases}
U_{\alpha} & \begin{array}{l} \text{if all roots have the same length,} \\ \alpha=\text{highest positive root,} \end{array} \\ \\ U_{\alpha_{\ell}}U_{\alpha_{s}} & \begin{array}{l} \text{if roots have different lengths,} \\ \alpha_{\ell}=\text{highest long root,} \\ \alpha_{s}=\text{highest short root,} \end{array}
\end{cases}
\end{align}
so in particular $\dim(Z(U))\in\{1,2\}$. During the proof of Proposition~\ref{pr:findfield}, we will fix a $B$-invariant subgroup $V\leq Z(U)$ of minimal dimension among those with $|\Gamma\cap V(\overline{K})|>1$; based on that choice, we shall define
\begin{align}\label{eq:rho}
\rho:=\begin{cases}
\alpha^{G} & \text{if } Z(U)=V=U_{\alpha}, \\ (\alpha^{G}_{\ell},\alpha^{G}_{s}) & \text{if } Z(U)=V=U_{\alpha_{\ell}}U_{\alpha_{s}}, \\ \alpha^{G}_{\ell} & \text{if } Z(U)=U_{\alpha_{\ell}}U_{\alpha_{s}} \text{ and } V=U_{\alpha_{\ell}}, \\ \alpha^{G}_{s} & \text{if } Z(U)=U_{\alpha_{\ell}}U_{\alpha_{s}} \text{ and } V=U_{\alpha_{s}}.
\end{cases}
\end{align}
In all cases, $\rho$ is a representation over $K$, and when $\dim(V)=2$ it can be seen as a representation over $K^{2}$ as well. By what we argued so far, we can bound the maximal degree of $\rho$.

\begin{proposition}\label{pr:degrho}
Let $G$ be a connected almost simple adjoint linear algebraic group defined over a field $K$, with $d=\dim(G)$ and $\iota=\mathrm{mdeg}(^{-1})$, and let $\rho$ be defined as in \eqref{eq:rho} (for an appropriate choice of $V$). Then $\rho:G\rightarrow\mathrm{GL}_{m}$ is a representation defined over $K$, and as such we have
\begin{align*}
m & \leq d, & \mathrm{mdeg}(\rho) & \leq\iota+1\leq d.
\end{align*}
\end{proposition}

\begin{proof}
The bound we obtained above for $\alpha^{G},\alpha^{G}_{s},\alpha^{G}_{\ell}$ becomes directly $\mathrm{mdeg}(\rho)\leq\iota+1$. We started the construction from a representation $G\leq\mathrm{GL}_{d}$ as in Proposition~\ref{pr:adjbounds}, so $\iota\leq d-1$ as well. Comparing \cite[Prop.~1.11]{Pin98} with the values of $|\Lambda_{w}/\Lambda_{r}|$ for each type (readable from \cite[App.~(A.9)]{Hum95b} or \cite[Table~9.2]{MT11}), we see that $\dim(\bar{\mathfrak{g}}_{s})+\dim(\bar{\mathfrak{g}}_{\ell})\leq\dim(\mathfrak{g})$.
\end{proof}

\section{Dimensional estimates}\label{se:dimest}

In this section we deal with dimensional estimates. Our work relies on the explicit estimates for $|A\cap V(K)|$ contained in \cite{BDH24}, which we first adapt so as to replace a set $A$ generating $G(K)$ with a subgroup $\Gamma\leq G(\overline{K})$ not necessarily equal to $G(K)$. This first result is given in Theorem~\ref{th:dimest}. Then we provide a few additional variations that will be useful later.

We start with the following lemma, which allows us to grow in dimension using a generic element and the almost simplicity of $G$.

\begin{lemma}\label{le:skewness}
Let $G\leq\mathrm{GL}_{n}$ be an almost simple linear algebraic group over a field $K$. Let $\iota = \mathrm{mdeg}(^{-1})$, the maximum degree of the inversion map. Let $V,V'$ be subvarieties of $G$ defined over $\overline{K}$, with $\dim(V)<\dim(G)$ and $\dim(V')>0$.

Then, for every $g\in G(\overline{K})$ outside a variety $W=\{x\in\mathrm{Mat}_{n}:F(x)=0\}$ with $\deg(F)\leq 1+\min\{\iota,\dim(V)\}$ and $G\not\subseteq W$, the variety $\overline{VgV'}$ has dimension $>\dim(V)$.
\end{lemma}

\begin{proof}
See~\cite[Lem.~4.1]{BDH24}.
\end{proof}

Now we write our main inductive step for Theorem~\ref{th:dimest}.

\begin{proposition}\label{pr:oberstair}
Let $G\leq\mathrm{GL}_{n}$ be a connected almost simple linear algebraic group of rank $r$ with $\dim(G)=d$ and $\deg(G)=D$, defined over a field $K$. Let $\Gamma\leq G(\overline{K})$ be a finite group. Then at least one of the following holds: 
\begin{enumerate}[(a)]
\item\label{th:dimest-small-oberstair} $|\Gamma|\leq \left( 2dD\right)^{d+1}$; 
\item\label{th:dimest-h-oberstair} $\Gamma\leq H(\overline{K})$ for some subgroup $H<G$ with $\dim(H)<d$ and $\deg(H)\leq \left( 2dD\right)^{d+1}$;
\item\label{th:dimest-ok-oberstair} for any irreducible subvariety $V\subsetneq G$ defined over $\overline{K}$ with $0<\dim(V)=d'<d$, there is an integer $\ell\leq d-d'+1$ and there are proper subvarieties $F_{1},\ldots,F_{\ell-1}$ of $V$ and a variety $E\subsetneq \mathrm{Mat}_n$
with
\begin{equation*}
|\Gamma \cap V(\overline{K})|^{\ell}\leq|\Gamma|\cdot\prod_{j=1}^{\ell-1}|\Gamma\cap F_{j}(\overline{K})|+\ell|\Gamma\cap E(\overline{K})||\Gamma\cap V(\overline{K})|^{\ell-1}
\end{equation*}
and $E$ not containing $V$, and satisfying the following properties as well: 
\begin{align}\label{eq:hodor}
\sum_{j=1}^{\ell-1}\dim(F_{j}) & =\ell d'-d, & \deg(F_{j}) & \leq 2^d \ell^{d-1} \deg(V)^{j+1},
& \deg(E) &\leq (2 \ell)^{d} \deg(V)^{\ell}.
\end{align}
\end{enumerate}
\end{proposition}

\begin{proof}
The statement and proof are very similar to~\cite[Prop.~4.4]{BDH24}, which differs from our result in that in lieu of $\Gamma$ it has a set $A$ that generates $G(K)$. The fact that $\langle A\rangle=G(K)$ was only needed in order to provide an escape argument, and obviously here we do not always have $\langle\Gamma\rangle=\Gamma=G(K)$. Hence, we have to replace the escape argument, which we are able to do by relying on Lemma~\ref{le:groupvar}.

Let us assume that $\Gamma$ does not satisfy the conditions~\eqref{th:dimest-small-oberstair} and~\eqref{th:dimest-h-oberstair}. Then, following Lemma~\ref{le:groupvar} we get that $\Gamma \nsubseteq W(\overline{K})$ for any subvariety $W$ of $G$ with $\deg(W) \leq 2dD$. We shall show that $\Gamma$ satisfies condition~\eqref{th:dimest-ok-oberstair}.

As in the statement, let $V\subsetneq G$ be any irreducible subvariety defined over $\overline{K}$ with $0<\dim(V)=d'<d$. Fix an integer $\ell$ and elements $g_{1},\ldots,g_{\ell-1}\in G(\overline{K})$ (to be chosen soon), and define
\begin{align*}
V_{1} & =V, & V_{j+1} & =\overline{V_{j}g_{j}V} \ \ \ \ \ (1\leq j<\ell).
\end{align*}   
For each $1\leq j\leq\ell$, let $d'_{j}=\dim(V_{j})$ and $D_{j}=\deg(V_{j})$. By Lemma~\ref{le:zarimdeg}\eqref{le:zarimdeg-im}, we have $D_{j}\leq\deg(V)^{j}j^{d'_j}$. We will show that the $g_{j}$ can be chosen so that they belong to $\Gamma$ and that we have $d'_{j+1}>d'_{j}$ at every step.

If $G$ is a connected almost simple linear algebraic group, by Lemma~\ref{le:skewness} we get $d'_{j+1}>d'_{j}$ for any choice of $g_{j}$ outside a variety $W_{j}$ with $\deg(W_{j})\leq 1+\min\{\iota,d'_{j}\}$ and $G\nsubseteq W_{j}$. Consider the subvariety $G \cap W_{j}$ of $G$, which is a proper subvariety of dimension $<d$ and of degree 
$$\deg (G \cap W_{j}) \leq D\deg(W_{j})\leq D(d+1) \leq 2dD.$$
Then, by what we said before, $\Gamma \nsubseteq \left( G \cap W_{j}\right)(\overline{K})$.

Hence, we can choose an element $g_{j}$ inside $\Gamma$ with $d'_{j+1} > d'_{j}$ at each step $j$, and let $\ell$ be the least index such that $V_{\ell}=G$. Clearly, $\ell\leq d-d'+1$, and $g_{1},\ldots,g_{\ell-1}\in \Gamma$. The rest of the proof follows, mutatis mutandis, the steps in the proof of~\cite[Prop.~4.4]{BDH24}. 
\end{proof}

\begin{theorem}\label{th:dimest}
Let $G\leq\mathrm{GL}_{n}$ be a connected almost simple linear algebraic group of rank $r$ with $\dim(G)=d$ and $\deg(G)=D$, defined over a field $K$, and let $\Gamma \leq G(\overline{K})$ be a finite subgroup. Then at least one of the following holds: 
\begin{enumerate}[(a)]
\item\label{th:dimest-small} $|\Gamma|\leq \left( 2dD\right)^{d+1}$; 
\item\label{th:dimest-h} $\Gamma\leq H(\overline{K})$ for some subgroup $H<G$ with $\dim(H)<d$ and $\deg(H)\leq \left( 2dD\right)^{d+1}$;
\item\label{th:dimest-ok} for any proper subvariety $V\subsetneq G$, we have 
\begin{align}\label{eq:boundC}
|\Gamma\cap V(\overline{K})|\leq C|\Gamma|^{\dim(V)/d}\quad \text{with} \quad C\leq \left( 2 d\deg(V)\right)^{d^{\dim(V)}}.
\end{align}
\end{enumerate}
\end{theorem}

\begin{proof}
The statement and proof are almost as in~\cite[Thm.~4.5]{BDH24}. The inductive step here is provided by Proposition~\ref{pr:oberstair}, which introduces cases~\eqref{th:dimest-small} and~\eqref{th:dimest-h}. After replacing $A$ with the group $\Gamma$ in the conclusion, we can ignore the exponent $C_{2}$ since $\Gamma=\langle\Gamma\rangle$, and the computation of $C_{1}$, which is denoted by $C$ here, is the same.
\end{proof}

The dimensional estimate of Theorem~\ref{th:dimest} can be extended to cover more general $G$. If one is not concerned about the correct exponent for $|\Gamma|$, this is quite easy to achieve in the group $G^{k}$ (the direct product of $k$ copies of $G$).

\begin{corollary}\label{co:gkbaddim}
Let $G\leq\mathrm{GL}_{n}$ be a connected almost simple linear algebraic group of rank $r$ with $\dim(G)=d$ and $\deg(G)=D$, defined over a field $K$, and let $\Gamma\leq G(\overline{K})$ be a finite subgroup. Then at least one of the following holds: 
\begin{enumerate}[(a)]
\item\label{co:gkbaddim-small} $|\Gamma|\leq(2dD)^{d+1}$;
\item\label{co:gkbaddim-h} $\Gamma\leq H(\overline{K})$ for some subgroup $H<G$ with $\dim(H)<d$ and $\deg(H)\leq (2dD)^{d+1}$;
\item\label{co:gkbaddim-ok} for any $k\geq 1$ and any proper subvariety $V\subsetneq G^{k}$, we have 
\begin{equation*}
|\Gamma^{k}\cap V(\overline{K})|\leq k(2d\deg(V))^{d^{d-1}}|\Gamma|^{k-\frac{1}{d}}.
\end{equation*}
\end{enumerate}
\end{corollary}

\begin{proof}
Apply the ``crumbling'' lemma given as~\cite[Lem.~4.3]{BDH24}, so that
\begin{equation*}
|\Gamma^{k}\cap V(\overline{K})|\leq k|\Gamma|^{k-1}|\Gamma\cap V'(\overline{K})|
\end{equation*}
for some subvariety $V'\subsetneq G$ with $\deg(V')\leq\deg(V)$, and then apply Theorem~\ref{th:dimest} to bound $|\Gamma\cap V'(\overline{K})|$.
\end{proof}

If we want to achieve the correct exponent for $|\Gamma^{k}\cap V(\overline{K})|$, we can set up an appropriate induction on both $k$ and $\dim(V)$. We do so in full generality in the following lemma, and then we show below an application to orbits that will yield a useful lower bound for centralizers of subsets of $\Gamma$.

\begin{lemma}\label{le:dimest-k}
Let $G\leq\mathrm{GL}_{n}$ be a connected almost simple linear algebraic group of rank $r$ with $\dim(G)=d$, $\deg(G)=D$ and $\mathrm{mdeg}(^{-1})=\iota$, defined over a field $K$, and let $\Gamma\leq G(\overline{K})$ be a finite subgroup. Then at least one of the following holds:
\begin{enumerate}[(a)]
\item\label{le:dimest-k-small} $|\Gamma|\leq(2dD)^{d+1}$;
\item\label{le:dimest-k-h} $\Gamma\leq H(\overline{K})$ for some subgroup $H<G$ with $\dim(H)<d$ and $\deg(H)\leq(2dD)^{d+1}$;
\item\label{le:dimest-k-ok} for any $k\geq 1$ and any proper subvariety $V\subsetneq G^{k}$, we have 
\begin{equation*}
|\Gamma^{k}\cap V(\overline{K})|\leq(2d\deg(V))^{kd^{\dim(V)}}|\Gamma|^{\frac{\dim(V)}{d}}.
\end{equation*}
\end{enumerate}
\end{lemma}

\begin{proof}
Assume that \eqref{le:dimest-k-small} and \eqref{le:dimest-k-h} do not hold. We prove the inequality in \eqref{le:dimest-k-ok} by applying a double induction on $k$ and $\dim(V)$. The base case $k=1$ follows from Theorem~\ref{th:dimest}, and the other base case $\dim(V)=0$ is vacuously true (since $|\Gamma^{k}\cap V(\overline{K})|\leq\deg(V)$). Assume that \eqref{le:dimest-k-ok} holds for all pairs $(k',V')$ that either have $k'<k$ and $\dim(V')\leq\dim(V)$ or have $k'\leq k$ and $\dim(V')<\dim(V)$. We may assume that $V$ is irreducible, since the bound is more than linear in $\deg(V)$.
    
Consider the projection map $\pi : V \rightarrow G^{k-1}$ defined by $$(x_1, \cdots, x_{k-1}, x_{k}) \mapsto (x_1, \cdots, x_{k-1}).$$
Since $V$ is irreducible, by \cite[Lem.~4.2]{BDH24}
\begin{equation}\label{eq:Gamma-k}
|\Gamma^{k}\cap V(\overline{K})|\leq|\Gamma^{k-1}\cap\overline{\pi(V)}(\overline{K})||\Gamma^{k}\cap W(\overline{K})|+|\Gamma^{k}\cap E(\overline{K})|,
\end{equation}
where $W=\pi^{-1}(y)\subsetneq G^{k}$ for some point $y\in\Gamma^{k-1}\cap\pi(V)$ with $\dim(W)=\dim(V)-\dim(\overline{\pi(V)})$, and where $E=\pi^{-1}(Z)\subsetneq V$ for some subvariety $Z\subsetneq\overline{\pi(V)}$ with $\deg(Z)\leq\mdeg(\pi)^{\dim(\overline{\pi(V)})-1}\deg(V)$. The projection $\pi$ has $\mdeg(\pi)=1$, so by Lemma~\ref{le:zarimdeg}\eqref{le:zarimdeg-im} we obtain the bounds
\begin{align*}
\deg(\overline{\pi(V)}) & \leq\deg(V), & \deg(W) & \leq\deg(V), & \deg(E) & \leq\deg(V)^{2}.
\end{align*}
We have also $\dim(E)<\dim(V)$ since $V$ is irreducible, so by induction on the dimension
\begin{align}\label{eq:dim-E}
|\Gamma^{k}\cap E(\overline{K})| & \leq (2d\deg(V)^{2})^{k d^{\dim(E)}}|\Gamma|^{\frac{\dim(E)}{d}}\leq (2d\deg(V))^{2k d^{\dim(V)-1}}|\Gamma|^{\frac{\dim(V)-1}{d}},
\end{align}
and by induction on $k$
\begin{equation}\label{eq:dim-piV}
|\Gamma^{k-1}\cap\overline{\pi(V)}(\overline{K})|\leq(2d\deg(V))^{(k-1)d^{\dim(\overline{\pi(V)})}} |\Gamma|^{\frac{\dim(\overline{\pi(V)})}{d}}.
\end{equation}

We now estimate $|\Gamma^{k}\cap V(\overline{K})|$ for all three different possibilities:
\begin{enumerate}
\item $\dim(\overline{\pi(V)})=0$.
\item $\dim(W)=0$.
\item $ 1 \leq \dim(W), \dim(\overline{\pi(V)}) \leq \dim(V)-1 $.
\end{enumerate}

Consider first $\dim(\overline{\pi(V)})=0$. Since $V$ is irreducible, $\overline{\pi(V)}$ is irreducible, therefore $\overline{\pi(V)}=\{y\}$ for some $y=(y_1,\cdots,y_{k-1})$. Hence $V=\pi^{-1}(y)$, which means that $V$ is in fact a variety of the form $\{y_1\} \times \cdots \times \{y_{k-1}\} \times V$ sitting inside a copy of $G$ itself. The result follows from Theorem~\ref{th:dimest}.

Now consider $\dim(W)=0$. In this case $\dim(\overline{\pi(V)})=\dim(V)$, and $W$ is a set of $\deg(W)$ points so $|\Gamma^{k}\cap W(\overline{K})|\leq\deg(W)$. Therefore, following~\eqref{eq:Gamma-k} along with~\eqref{eq:dim-E} and \eqref{eq:dim-piV}, for $k\geq 1$, $\dim(V)\geq 1$, and $d\geq 3$, we write that
\begin{align*}
|\Gamma^{k}\cap V(\overline{K})| & \leq (2d\deg(V))^{(k-1)d^{\dim(V)}+1}|\Gamma|^{\frac{\dim(V)}{d}}+(2d\deg(V))^{2kd^{\dim(V)-1}}|\Gamma|^{\frac{\dim(V)-1}{d}} \\
& \leq  (2d \deg(V))^{kd^{\dim(V)}-2} |\Gamma|^{\frac{\dim(V)}{d}}+ (2d\deg(V))^{k d^{\dim(V)}-1}|\Gamma|^{\frac{\dim(V)}{d}} \\
 & \leq (2d\deg(V))^{kd^{\dim(V)}}|\Gamma|^{\frac{\dim(V)}{d}}.
\end{align*}

Finally consider the case $1\leq\dim(W),\dim(\overline{\pi(V)})\leq\dim(V)-1$. Applying induction on the dimension, we know that
\begin{align}\label{eq:dim-W}
|\Gamma^k\cap W(\overline{K})|\leq(2d\deg(V))^{kd^{\dim(W)}}|\Gamma|^{\frac{\dim(W)}{d}}.
\end{align}
Following \eqref{eq:dim-piV} and~\eqref{eq:dim-W}, we find that
\begin{align}
|\Gamma^{k-1}\cap\overline{\pi(V)}(\overline{K})||\Gamma^{k}\cap W(\overline{K})| & \leq  (2 d\deg(V))^{kd^{\dim(W)}+(k-1)d^{\dim(\overline{\pi(V)})}} |\Gamma|^{\frac{\dim(W)+\dim(\overline{\pi(V)})}{d}} \nonumber \\
& \leq  (2 d\deg(V))^{2 k d^{\dim(V)-1}} |\Gamma|^{\frac{\dim(V)}{d}}, \label{eq:dim-W+piV}
\end{align}
since $\dim(W) +\dim(\overline{\pi(V)})=\dim(V)$. Using~\eqref{eq:dim-E} and~\eqref{eq:dim-W+piV} in~\eqref{eq:Gamma-k}, we get 
\begin{align*}
|\Gamma^{k}\cap V(\overline{K})| & \leq 2(2d\deg(V))^{2kd^{\dim(V)-1}}|\Gamma|^{\frac{\dim(V)}{d}}\\
& \leq (2d\deg(V))^{kd^{\dim(V)}}|\Gamma|^{\frac{\dim(V)}{d}},
\end{align*}
since $d\geq 3$ and $\dim(V)\geq 1$.
\end{proof}

Thanks to the previous lemma, we are able to obtain not only an upper bound but also a lower bound for centralizers. This is due essentially to the {\em orbit-stabilizer theorem}: if a finite group $G$ acts on a set $X$ and $C_{G}(x),Gx$ are respectively the stabilizer and the orbit of a point $x\in X$ under this action, then $|C_{G}(x)||Gx|=|G|$. If $G$ is an algebraic group, as in our case, in certain cases an alternative version of the theorem still holds. If $G$ is a connected semisimple algebraic group acting on itself by conjugation, $C_{G}(x)$ is the centralizer as defined in Section~\ref{se:linalg-basic}, and $\mathrm{Cl}_{G}(x)$ (the {\em conjugacy class} of $x$) is the image of $G$ under the map $\sigma_{x}:G\rightarrow G$ given by $\sigma_{x}(g)=gxg^{-1}$, then
\begin{equation}\label{eq:osthm}
\dim(C_{G}(x))+\dim(\overline{\mathrm{Cl}_{G}(x)})=\dim(G)
\end{equation}
(see for instance \cite[\S 1.5]{Hum95a}). Moreover, \eqref{eq:osthm} still holds if $G$ acts by conjugation on $G^{k}$ as follows: $\left( g, (x_1, \cdots, x_k)\right) \mapsto \left(gx_1g^{-1}, \cdots, gx_kg^{-1} \right)$. In fact, we can see $C_{G}(x)$ as the fibre of $x$ through $\sigma_{x}:G\rightarrow G^{k}$, and then \eqref{eq:osthm} follows from fundamental results of algebraic geometry (see \cite[\S I.8]{Mumford}, used also in \cite[\S 4]{BDH24} and when we say $\dim(W)=\dim(V)-\dim(\overline{\pi(V)})$ in the proof of Lemma~\ref{le:dimest-k}) and from the fact that every fibre through $\sigma_{x}$ must have the same dimension (as they are images of each other through appropriate conjugation maps).

\begin{corollary}\label{co:estcentr}
Let $G\leq\mathrm{GL}_{n}$ be a connected almost simple linear algebraic group of rank $r$ with $\dim(G)=d$, $\deg(G)=D$ and $\mathrm{mdeg}(^{-1})=\iota$, defined over a field $K$, and let $\Gamma\leq G(\overline{K})$ be a finite subgroup. Then at least one of the following holds:
\begin{enumerate}[(a)]
\item\label{co:estcentr-small} $|\Gamma|\leq(2dD)^{d+1}$;
\item\label{co:estcentr-h} $\Gamma\leq H(\overline{K})$ for some subgroup $H<G$ with $\dim(H)<d$ and $\deg(H)\leq(2dD)^{d+1}$;
\item\label{co:estcentr-ok} for any subset $\Lambda\subseteq\Gamma$, the centralizer $C_{G}(\Lambda)$ satisfies the bounds
\begin{equation*}
\frac{1}{\varphi(d-d')}|\Gamma|^{\frac{d'}{d}}\leq|\Gamma\cap C_{G}(\Lambda)(\overline{K})|\leq \varphi(d')|\Gamma|^{\frac{d'}{d}}
\end{equation*}
\end{enumerate}
where $d'=\dim(C_{G}(\Lambda))$ and $\varphi(x)=(2dD(\iota+1))^{x(d+1)d^{x}}$.
\end{corollary}

This result corresponds to \cite[Thm.~6.2]{LP11}. We are more pedantic with the constants, using the function $\varphi(x)$ instead of a unique constant $c_{0}$ as in \cite{LP11}, since otherwise the tower in the bound of Theorem~\ref{th:main}\eqref{th:main-small} has one more floor (due to the factorial in \eqref{eq:gfgfup}).

\begin{proof}
Our strategy is to use the upper bounds coming from the previous theorems on both centralizers and conjugacy classes, and transform the upper bound for the latter into a lower bound for the former via the orbit-stabilizer theorem and \eqref{eq:osthm}. Since we have a set $\Lambda$ instead of a single element $x$ we must work in the direct product $G^{k}$, which is why we needed Lemma~\ref{le:dimest-k} and we had to explain why \eqref{eq:osthm} works for the action of $G$ on $G^{k}$.

Assume that \eqref{co:estcentr-small} and \eqref{co:estcentr-h} do not hold. By Corollary~\ref{co:centralizer} we may assume that $\Lambda=\{\gamma_1,\cdots,\gamma_k\} \subseteq G(\overline{K})$ with $k\leq d+1$, and also $\deg(C_{G}(\Lambda))\leq D$. Call $\vec{\lambda}=(\gamma_1,\cdots,\gamma_k)\in G^{k}$. $G$ acts on $G^{k}$ by conjugation as follows: $\left( g, (g_1, \cdots, g_k)\right) \mapsto \left(gg_1g^{-1}, \cdots, gg_kg^{-1} \right)$. Under this action $C_{G}(\vec{\lambda})=C_{G}(\Lambda)$, and following Theorem~\ref{th:dimest}\eqref{th:dimest-ok}
\begin{align}\label{eq:centralizer}
|\Gamma \cap C_{G}(\Lambda)(\overline{K})| & \leq C_1 |\Gamma|^{\frac{\dim(C_{G}(\Lambda))}{d}}, & C_1 & \leq (2d \deg(C_{G}(\Lambda)))^{d^{\dim(C_{G}(\Lambda))}},
\end{align}
so that $C_{1}\leq\varphi(\dim(C_{G}(\Lambda)))$. Similarly, Lemma~\ref{le:dimest-k} gives
\begin{align}\label{eq:conjclass}
|\Gamma^{k}\cap\overline{\mathrm{Cl}_{G}(\vec{\lambda})}(\overline{K})| & \leq C_2 |\Gamma|^{\frac{\dim(\overline{\mathrm{Cl}_{G}(\vec{\lambda}))}}{d}}, & C_2 & \leq (2d\deg(\overline{\mathrm{Cl}_{G}(\vec{\lambda}))})^{k d^{\dim(\overline{\mathrm{Cl}_{G}(\vec{\lambda})})}}.
\end{align}
By \eqref{eq:osthm} we have $\dim(C_{G}(\Lambda))+\dim(\overline{\mathrm{Cl}_{G}(\vec{\lambda})})=d$, and by Lemma~\ref{le:zarimdeg}\eqref{le:zarimdeg-im} $\deg(\overline{\mathrm{Cl}_{G}(\vec{\lambda})})\leq D(\iota+1)^{\dim(\overline{\mathrm{Cl}_{G}(\vec{\lambda})})}$, so that $C_{2}\leq\varphi(\dim(\overline{\mathrm{Cl}_{G}(\vec{\lambda})}))$.

To make \eqref{eq:conjclass} into a lower bound, we now need to pass to stabilizers and orbits in $\Gamma$ itself. The equality $C_{\Gamma}(\Lambda)=\Gamma\cap C_{G}(\Lambda)(\overline{K})$ is trivial by definition, and since $\Lambda\subseteq\Gamma$ we also have the inclusion $\mathrm{Cl}_{\Gamma}(\vec{\lambda})\subseteq\Gamma^{k}\cap\overline{\mathrm{Cl}_{G}(\vec{\lambda})}(\overline{K})$. By the orbit-stabilizer theorem then
\begin{equation}\label{eq:orbit}
|\Gamma\cap C_{G}(\Lambda)(\overline{K})|=\frac{|\Gamma|}{|\mathrm{Cl}_{\Gamma}(\vec{\lambda})|} \geq\frac{1}{C_2}|\Gamma|^{1-\frac{\dim(\overline{\mathrm{Cl}_{G}(\vec{\lambda}))}}{d}}=\frac{1}{C_2}|\Gamma|^{\frac{\dim(C_{G}(\Lambda))}{d}}.
\end{equation}
 
We get the desired inequalities in \eqref{co:estcentr-ok} by combining~\eqref{eq:centralizer} and~\eqref{eq:orbit}.
\end{proof}

\section{Finding a simple group of Lie type}\label{se:findgf}

The present section is devoted to giving an explicit version of \cite[Thm.~0.5]{LP11}. In brief, we are in a situation where $\Gamma$ is a finite subgroup of $G(\overline{K})$, where $G$ is already assumed to be connected almost simple adjoint. Under these conditions, we shall conclude that either $\Gamma$ is not ``sufficiently general'' (meaning that either $|\Gamma|$ is bounded or $\Gamma$ is contained in $H(\overline{K})$ with $H$ of bounded degree and strictly smaller dimension) or $\Gamma$ is, up to a small index, a finite simple group of Lie type in the same characteristic as $K$ (meaning that $[G^{F},G^{F}]\leq\Gamma\leq G^{F}$ for some Steinberg endomorphism $F$).

All the results of this section work under the same assumptions, except that we may need to use two different explicit versions of the ``sufficiently general'' condition. We collect the assumptions here and reference the appropriate choices throughout the section to make the statements less cumbersome.


\begin{assumption}
$G\leq\mathrm{GL}_{n}$ is a connected almost simple adjoint linear algebraic group of rank $r$ defined over an arbitrary field $K$, with $d=\dim(G)$, $D=\deg(G)$, and $\iota=\mathrm{mdeg}(^{-1})$. $\Gamma\leq G(\overline{K})$ is a finite subgroup, either satisfying

\textit{(Assumption~\customlabel{ass-large1}{A1})}\ \ $|\Gamma|>(2dDrn\iota)^{(2dDr\iota)^{10d^{4}}}$, and

\textit{(Assumption~\customlabel{ass-h1}{A2})}\ \ $\Gamma\not\leq H(\overline{K})$ for any linear algebraic subgroup $H\lneq G$ with
\begin{align*}
\dim(H)<d \quad \text{and} \quad \deg(H)\leq(2dD)^{4d},
\end{align*}
or satisfying

\textit{(Assumption~\customlabel{ass-large2}{B1})}\ \ $|\Gamma|>(2dDrn\iota)^{(2dDr\iota)^{11d^{4}}}$, and

\textit{(Assumption~\customlabel{ass-h2}{B2})}\ \ $\Gamma\not\leq H(\overline{K})$ for any linear algebraic subgroup $H\lneq G$ with
\begin{align*}
\dim(H)<d \quad \text{and} \quad \deg(H)\leq(2dDr)^{4d^{2}}.
\end{align*}
\end{assumption}

The reader should not be alarmed by the salad of letters appearing in the bound for $|\Gamma|$: what is really important is the number of exponential floors we use, and we include all the parameters with the sole intention of covering the many ways in which the assumption plays its role in the rest of the paper. In the final theorem, everything can be expressed just in terms of the rank.

\subsection{Finding the finite field $\mathbb{F}_{q}$}

The first goal is to find the ``correct'' field $\mathbb{F}_{q}$ that underlies the structure of $\Gamma$: if the final objective is to have $[G^{F},G^{F}]\leq\Gamma\leq G^{F}$ for some Steinberg endomorphism $F$ (as defined in Section~\ref{se:steinberg}), we need to know to what field the endomorphism is referred.

First, we shall find a regular unipotent element $u\in\Gamma$. Its existence proves incidentally that $\mathrm{char}(K)\neq 0$, since otherwise any nontrivial unipotent element generates an infinite group. Then we shall use $u$ to find a unipotent algebraic subgroup $V$ for which $\Gamma\cap V(\overline{K})$ is isomorphic (as a group) to $\mathbb{F}_{q}$. The group $V$ is essentially made of ``minimal'' elements: for instance, if $G=\mathrm{PGL}_{n}$ and $u$ is upper triangular, $V$ is the variety of unipotent elements having their only nonzero off-diagonal entry in the upper right corner; then, the values taken in that entry by elements of $\Gamma$ form the field $\mathbb{F}_{q}$.

Following \cite[p.~1129]{LP11}, we call $\Lambda\subseteq G(\overline{K})$ a {\em toric subset} if it is contained in an algebraic torus of $G$. The preliminary result below shows that centralizers of toric subsets of $\Gamma$ contain mostly regular semisimple elements of $\Gamma$.

\begin{lemma}\label{le:rss}
Let $G,n,r,K,d,D,\iota,\Gamma$ be as in Assumption~\ref{ass-large1}--\ref{ass-h1}. Then, for any toric subset $\Lambda\subseteq\Gamma$ we have
\begin{equation*}
|(\Gamma\cap C_{G}(\Lambda)^{\mathrm{o}})^{\mathrm{rss}}|\geq\left(1-\frac{1}{2(2r)^{r}}\right)|\Gamma\cap C_{G}(\Lambda)^{\mathrm{o}}|.
\end{equation*}
\end{lemma}

\begin{proof}
We follow \cite[Prop.~7.2]{LP11}. For $g\in G(\overline{K})$, by \eqref{eq:mdegad} the matrix $\mathrm{Ad}_{G,g}\in\mathrm{Gl}(\mathfrak{g})(\overline{K})$ given by the adjoint representation has entries of degree $\leq\iota+1$ in the entries of $g$. We know that $g$ is not regular semisimple when the characteristic polynomial of $\mathrm{Ad}_{g}$ (call it $p_{g}(t)$) is divisible by $(t-1)^{r+1}$; using Hasse derivatives, this condition becomes $p_{g}^{(r)}(1)=0$, which is an equation of degree $\leq d(\iota+1)$ in the entries of $g$. Let $V$ be the variety of such $g$.

Since $\Lambda$ is toric, let $T$ be a maximal torus containing $\Lambda$. Every torus is connected \cite[\S 8.5]{Bor91} and its elements commute with all elements of $\Lambda$ by definition, therefore $C_{G}(\Lambda)^{\mathrm{o}}\supseteq T$. Since regular semisimple elements are dense in $T$ (see \cite[\S 2.3]{Hum95a}), the intersection $V\cap T$ is proper inside $T$, so $V\cap C_{G}(\Lambda)^{\mathrm{o}}\subsetneq C_{G}(\Lambda)^{\mathrm{o}}$ as well, and in particular $\dim(V\cap C_{G}(\Lambda)^{\mathrm{o}})\leq\dim(C_{G}(\Lambda)^{\mathrm{o}})-1$.

We also have $\deg(V\cap C_{G}(\Lambda)^{\mathrm{o}})\leq\deg(V)\deg(C_{G}(\Lambda))\leq d(\iota+1)D$ by Corollary~\ref{co:centralizer}. By Assumption~\ref{ass-large1}--\ref{ass-h1}, Theorem~\ref{th:dimest}\eqref{th:dimest-ok}, and Corollary~\ref{co:estcentr}\eqref{co:estcentr-ok}, 
\begin{align*}
|\Gamma\cap(V\cap C_{G}(\Lambda)^{\mathrm{o}})(\overline{K})| & \leq (2d^{2}D(\iota+1))^{d^{d-1}}|\Gamma|^{\frac{\dim(C_{G}(\Lambda)^{\mathrm{o}})-1}{d}} \\
 & \leq\frac{1}{(2dD(d+1))^{(d+1)d^{d+1}}\cdot 2D(2r)^{r}}|\Gamma|^{\frac{\dim(C_{G}(\Lambda)^{\mathrm{o}})}{d}} \\
 & \leq\frac{1}{2D(2r)^{r}}|\Gamma\cap C_{G}(\Lambda)(\overline{K})|\leq\frac{1}{2(2r)^{r}}|\Gamma\cap C_{G}(\Lambda)^{\mathrm{o}}(\overline{K})|,
\end{align*}
using the obvious equality $\dim(C_{G}(\Lambda))=\dim(C_{G}(\Lambda)^{\mathrm{o}})$ and the fact that by Corollary~\ref{co:centralizer} we have $|\Gamma\cap C_{G}(\Lambda)(\overline{K})|\leq\deg(C_{G}(\Lambda))|\Gamma\cap C_{G}(\Lambda)^{\mathrm{o}}(\overline{K})|\leq D|\Gamma\cap C_{G}(\Lambda)^{\mathrm{o}}(\overline{K})|$.
\end{proof}

Under the same conditions, the same bound as in Lemma~\ref{le:rss}, namely
\begin{equation}\label{eq:rss-mts}
|\Theta^{\mathrm{rss}}|\geq\left(1-\frac{1}{2(2r)^{r}}\right)|\Theta|,
\end{equation}
holds for any maximal toric subgroup $\Theta\leq\Gamma\cap C_{G}(\Lambda)^{\mathrm{o}}$ as well: see \cite[Prop.~7.3]{LP11}.

Now we prove that there is a regular unipotent $u\in\Gamma$.

\begin{proposition}\label{pr:run}
Let $G,n,r,K,d,D,\iota,\Gamma$ be as in Assumption~\ref{ass-large1}--\ref{ass-h1}.

Then $|\Gamma^{\mathrm{run}}|\geq\frac{1}{2}|\Gamma^{\mathrm{un}}|\geq 1$ and $\mathrm{char}(K)\neq 0$.
\end{proposition}

\begin{proof}
We follow the proof leading to \cite[Cor.~7.10]{LP11}; we skip the theoretical details, focusing on the explicit bounds.

Fix any toric subset $\Lambda\subseteq\Gamma$, let $\mathcal{T}(\Lambda)$ be the set of all maximal toric subgroups inside $\Gamma\cap C_{G}(\Lambda)^{\mathrm{o}}(\overline{K})$, and pick a set $\mathcal{S}(\Lambda)\subseteq\mathcal{T}(\Lambda)$ of representatives of $(\Gamma\cap C_{G}(\Lambda)^{\mathrm{o}}(\overline{K}))$-conjugacy classes of elements of $\mathcal{T}(\Lambda)$.
Using Lemma~\ref{le:rss}, it is shown in the proof of \cite[Prop.~7.4]{LP11} that
\begin{equation}\label{eq:torsumtrap}
1-\frac{1}{2(2r)^{r}}\leq\sum_{\Theta\in\mathcal{S}(\Lambda)}\frac{1}{[N_{\Gamma\cap C_{G}(\Lambda)^{\mathrm{o}}(\overline{K})}(\Theta):\Theta]}\leq\frac{1}{1-\frac{1}{2(2r)^{r}}}.
\end{equation}
Furthermore, each denominator in the sum above divides the size of the Weyl group of $G$, so by \eqref{eq:weilg} the sum in \eqref{eq:torsumtrap} must be equal to $1$. This implies
\begin{align}\label{eq:torsumexact}
\sum_{\Theta\in\mathcal{T}(\Lambda)}|\Theta| & =|\Gamma\cap C_{G}(\Lambda)^{\mathrm{o}}(\overline{K})|, & \sum_{\Theta\in\mathcal{T}(\Lambda)}1 & =|(\Gamma\cap C_{G}(\Lambda)^{\mathrm{o}}(\overline{K}))^{\mathrm{un}}|.
\end{align}
For all $\Theta\in\mathcal{T}(\Lambda)$, we also see from the proof of \cite[Prop.~7.3]{LP11} that 
\begin{align}\label{eq:thetafacts}
\Theta & =\Gamma\cap C_{G}(\Theta)^{\mathrm{o}}(\overline{K}), & \dim(C_{G}(\Theta)^{\mathrm{o}}) & =r.
\end{align}
We can apply the second equality of \eqref{eq:torsumexact} to $\Lambda=\emptyset$, for which $C_{G}(\emptyset)^{\mathrm{o}}=G^{\mathrm{o}}=G$, then the first equality of \eqref{eq:thetafacts}, then the upper bound of Corollary~\ref{co:estcentr} to $C_{G}(\Theta)\supseteq C_{G}(\Theta)^{\mathrm{o}}$, and finally the second equality of \eqref{eq:thetafacts} and the fact that the sum in \eqref{eq:torsumtrap} is equal to $1$. We conclude that
\begin{align}
|\Gamma^{\mathrm{un}}| & \overset{\text{\eqref{eq:torsumexact}}}{=}\sum_{\Theta\in\mathcal{T}(\emptyset)}1=\sum_{\Theta\in\mathcal{S}(\emptyset)}\frac{|\Gamma|}{|N_{\Gamma}(\Theta)|} \overset{\text{\eqref{eq:thetafacts}}}{=}\sum_{\Theta\in\mathcal{S}(\emptyset)}\frac{|\Gamma|}{[N_{\Gamma}(\Theta):\Theta]\cdot|\Gamma\cap C_{G}(\Theta)^{\mathrm{o}}|} \nonumber \\
 & \overset{\text{Cor.~\ref{co:estcentr}}}{\geq} \!\! \sum_{\Theta\in\mathcal{S}(\emptyset)}\frac{|\Gamma|^{1-\frac{\dim(C_{G}(\Theta)^{\mathrm{o}})}{d}}}{[N_{\Gamma}(\Theta):\Theta]\cdot(2dD(\iota+1))^{(d+1)d^{d+1}}}\overset{\text{\eqref{eq:torsumtrap}--\eqref{eq:thetafacts}}}{=}\frac{|\Gamma|^{1-\frac{r}{d}}}{(2dD(\iota+1))^{(d+1)d^{d+1}}}. \label{eq:gunall}
\end{align}

Now we bound $|\Gamma^{\mathrm{un}}\setminus\Gamma^{\mathrm{run}}|=|\Gamma\cap(G^{\mathrm{un}}\cap G^{\mathrm{irr}})(\overline{K})|$. As mentioned in Section~\ref{se:linalg-basic}, $G^{\mathrm{un}}$ is an irreducible variety of dimension $d-r$, and by Corollary~\ref{co:inters}\eqref{co:inters-extra} it has degree
\begin{equation*}
\deg(G^{\mathrm{un}})=\deg(G\cap\{x\in\mathrm{Mat}_{n+1}:(x|_{n\times n}-\mathrm{Id}_{n})^{n}=0\})\leq Dn^{d+1}.
\end{equation*}
Then, $G^{\mathrm{un}}\cap G^{\mathrm{irr}}$ is a proper subvariety of $G^{\mathrm{un}}$ \cite[\S 4.13]{Hum95a}, and it can be defined as the set of $g\in G^{\mathrm{un}}$ for which, if $f:G\times G^{\mathrm{un}}\rightarrow G^{\mathrm{un}}$ is the map given by $f(x,y)=xyx^{-1}$, the fibre $f^{-1}(g)$ is larger than the generic one \cite[\S 1.4]{Hum95a}. By Proposition~\ref{pr:degexc}, there is some variety $Z$ with $G^{\mathrm{un}}\cap G^{\mathrm{irr}}\subseteq Z\subsetneq G^{\mathrm{un}}$ (so in particular $\dim(Z)<d-r$) and $\deg(Z)\leq(\iota+2)^{d-r-1}D^{2}n^{d+1}$. Then, Theorem~\ref{th:dimest}\eqref{th:dimest-ok} gives
\begin{equation}\label{eq:gunirr}
|\Gamma^{\mathrm{un}}\setminus\Gamma^{\mathrm{run}}|\leq|\Gamma\cap Z(\overline{K})|\leq\left(2d(\iota+2)^{d-r-1}D^{2}n^{d+1}\right)^{d^{d-r-1}}|\Gamma|^{1-\frac{r+1}{d}},
\end{equation}
since the alternatives \eqref{th:dimest-small} and \eqref{th:dimest-h} of Theorem~\ref{th:dimest} do not hold by Assumption~\ref{ass-large1}--\ref{ass-h1}. Combining \eqref{eq:gunall} and \eqref{eq:gunirr} with the condition on $|\Gamma|$ coming from  Assumption~\ref{ass-large1}, we obtain the claim on $\Gamma^{\mathrm{run}}$.

Finally, the existence of a regular unipotent element implies that $\mathrm{char}(K)$ divides its order, so $\mathrm{char}(K)\neq 0$ (see \cite[Cor.~7.11]{LP11}).
\end{proof}

We conclude the subsection by finding the desired algebraic subgroup $V$ of minimal unipotent elements. In the following, the normalizer $N_{\Gamma}(V)=\{\gamma\in\Gamma:\gamma V=V\gamma\}$ is a subgroup of $\Gamma\leq G(\overline{K})\leq\mathrm{GL}_{n}(\overline{K})$ with $\mathrm{char}(K)=p$, so it makes sense to talk about the group ring $\mathbb{F}_{p}[N_{\Gamma}(V)]$.

\begin{proposition}\label{pr:findfield}
Let $G,n,r,K,d,D,\iota,\Gamma$ be as in Assumption~\ref{ass-large1}--\ref{ass-h1}.

Then $\mathrm{char}(K)=p>0$, there is some $q=p^{e}$, and there is some abelian unipotent subgroup $V=V^{(G)}\leq G$ of dimension $1\leq\dim(V)\leq\min\{2,r\}$ and degree $\deg(V)\leq D$ such that $\Gamma\cap V(\overline{K})\simeq\mathbb{F}_{q}$ (as abelian groups), $\mathbb{F}_{q}\leq K^{\dim(V)}$ (as $\mathbb{F}_{p}$-algebras), $\mathbb{F}_{q}\simeq\mathbb{F}_{p}[N_{\Gamma}(V)]$ (as rings), and
\begin{align*}
\frac{1}{\varphi(d-1)}|\Gamma|^{\frac{\dim(V)}{d}} & \leq|\Gamma\cap V(\overline{K})|=q\leq\varphi(2)|\Gamma|^{\frac{\dim(V)}{d}}, & \varphi(x) & =(2dD(\iota+1))^{x(d+1)d^{x}}.
\end{align*}
\end{proposition}

The function $\varphi(x)$ above is the same as in Corollary~\ref{co:estcentr}.

\begin{proof}
We follow the proof leading to \cite[Thm.~8.17]{LP11}, skipping the theoretical details.

By Proposition~\ref{pr:run}, $\mathrm{char}(K)=p>0$ and there exists some $u\in\Gamma^{\mathrm{run}}$. Pick a Borel subgroup $B$ of $G$ so that its unipotent part $U$ contains $u$: $\Gamma\cap U(\overline{K})$ is a normal Sylow $p$-subgroup of $\Gamma\cap B(\overline{K})$, so by the Schur--Zassenhaus theorem $\Gamma\cap B(\overline{K})$ is a semidirect product of $\Gamma\cap U(\overline{K})$ and the corresponding quotient; then, either \cite[\S 19.4, Prop.~(a)]{Hum95b} or \cite[Thm.~10.6(5)]{Bor91} guarantee that there is a maximal torus $T$ in $B$ such that
\begin{equation}\label{eq:but}
\Gamma\cap B(\overline{K})=(\Gamma\cap U(\overline{K}))\rtimes(\Gamma\cap T(\overline{K})).
\end{equation}

First, we need a lower bound on $|\Gamma\cap T(\overline{K})|$: the torus $T$ is also a maximal torus of $G$ \cite[\S 21.3, Cor.~A]{Hum95b} and is its own centralizer \cite[\S 26.2, Cor.~A(b)]{Hum95b}, but it may not necessarily be the centralizer of a subset of $\Gamma$, so we cannot apply Corollary~\ref{co:estcentr}\eqref{co:estcentr-ok} directly. We follow instead the proof of \cite[Lem.~8.7]{LP11}. By \eqref{eq:but} and the orbit-stabilizer theorem, we know that
\begin{equation}\label{eq:three-87}
|\Gamma\cap T(\overline{K})|=|\Gamma\cap C_{G}(u)(\overline{K})|\cdot\frac{|\mathrm{Cl}_{\Gamma\cap B(\overline{K})}(u)|}{|(\Gamma\cap U(\overline{K}))^{\mathrm{run}}|}\cdot\frac{|(\Gamma\cap U(\overline{K}))^{\mathrm{run}}|}{|\Gamma\cap U(\overline{K})|}
\end{equation}
since $C_{G}(u)=C_{B}(u)$ for regular unipotent elements in semisimple groups \cite[p.~54]{Hum95a}. For the first term on the right-hand side of \eqref{eq:three-87} we can use now the lower bound in Corollary~\ref{co:estcentr}\eqref{co:estcentr-ok}, since Assumption~\ref{ass-large1}--\ref{ass-h1} forbids the other cases of Corollary~\ref{co:estcentr}. Without loss of generality, we may assume that $u$ has the largest possible class $\mathrm{Cl}_{\Gamma\cap B(\overline{K})}(u)$ and follow \cite[Thm.~6.6 and Lem.~8.7]{LP11} to deal with the second term: then, this term is bounded from below by the product of the inverses of two constants $C$ as in Theorem~\ref{th:dimest}\eqref{th:dimest-ok}, one for $V=C_{G}(u)$ (with $\deg(V)\leq D$ by Corollary~\ref{co:centralizer}) and one for $V=\overline{\mathrm{Cl}_{G}(u)}$ (with $\deg(V)\leq D(\iota+1)^{d-r}$ by Lemma~\ref{le:zarimdeg}\eqref{le:zarimdeg-im}), since Assumption~\ref{ass-large1}--\ref{ass-h1} forbids the other cases of Theorem~\ref{th:dimest}. Lastly, the third term of \eqref{eq:three-87} can be bounded using \cite[Lem.~8.6]{LP11}: the assumption ``$\Gamma$ is sufficiently general'' therein is fulfilled when $U^{\mathrm{run}}\neq\emptyset$, which we have. Putting everything together, \eqref{eq:three-87} yields
\begin{equation}\label{eq:boundtlow}
|\Gamma\cap T(\overline{K})|\geq\frac{|\Gamma|^{\frac{r}{d}}}{(2dD(\iota+1))^{(d-r)(d+1)d^{d-r}}}\cdot\frac{1}{(2dD)^{d^{r}}(2dD(\iota+1)^{d-r})^{d^{d-r}}}\cdot\frac{1}{2^{r^{r}}}
\end{equation}

Next, following the steps of \cite[Lem.~8.9]{LP11}, we prove that $Z(U)$ is the centralizer of some subset of $\Gamma$ so as to use Corollary~\ref{co:estcentr} on it. We have a lower bound for $|\Gamma\cap T(\overline{K})|$ in \eqref{eq:boundtlow}, and now we look for an appropriate upper bound. If $f:T\rightarrow U$ is the conjugation map $f(t)=tut^{-1}$, then for any given element $g\in C_{G}(f(\Gamma\cap T(\overline{K})))$ we get $\Gamma\cap T(\overline{K})\subseteq f^{-1}(C_{U}(g))(\overline{K})$, and in particular
\begin{equation}\label{eq:tft}
|\Gamma\cap T(\overline{K})|\leq|\Gamma\cap f^{-1}(C_{U}(g))(\overline{K})|.
\end{equation}
On the other hand,
\begin{align*}
\deg(f^{-1}(C_{U}(g))) & \leq\deg(T)\deg(C_{U}(g))\mathrm{mdeg}(f)^{\dim(C_{U}(g))} \\
 & \leq\deg(T)\deg(U)\mathrm{mdeg}(f)^{\dim(U)}\leq D^{2}(\iota+1)^{d-r}
\end{align*}
using Lemma~\ref{le:zarimdeg}\eqref{le:zarimdeg-pre}, the inequality $\deg(C_{U}(g))\leq\deg(U)$ (proved as in Corollary~\ref{co:centralizer}), Lemma~\ref{le:tudeg}\eqref{le:tudeg-t}--\eqref{le:tudeg-u}, and $\dim(U)\leq d-r$ (see for instance \cite[\S 4.2]{Hum95a}). Therefore, calling $r'=\dim(f^{-1}(C_{U}(g)))$, we have
\begin{equation}\label{eq:boundthigh}
|\Gamma\cap f^{-1}(C_{U}(g))(\overline{K})|\leq(2dD^{2}(\iota+1)^{d-r})^{d^{r'}}|\Gamma|^{\frac{r'}{d}}
\end{equation}
by Theorem~\ref{th:dimest}\eqref{th:dimest-ok}. If $r'<\dim(T)=r$, Assumption~\ref{ass-large1} becomes incompatible with \eqref{eq:boundtlow}--\eqref{eq:tft}--\eqref{eq:boundthigh}, so we must have $r'=\dim(T)$. This implies $f^{-1}(C_{U}(g))=T$ since $T$ is irreducible by definition, and thus $f(T)\subseteq C_{U}(g)$ for any $g\in C_{G}(f(\Gamma\cap T(\overline{K})))$. Since $U$ is also irreducible and $f(T)$ is made of regular elements, the image $f(T)/[U,U]$ of $f(T)\subseteq U$ in the quotient $U/[U,U]$ is the quotient's unique open $T$-orbit (see the description in terms of root subgroups in \cite[\S 4.1]{Hum95a}). Thus $f(T)/[U,U]$ generates $U/[U,U]$ as an algebraic group, and because $U$ is nilpotent similarly $f(T)$ generates $U$. Hence, since $f(T)\subseteq C_{U}(g)$ and the latter is an algebraic group too, we have $U=C_{U}(g)$, or in other words $g\in Z(U)$. By our choice of $g$ then $C_{G}(f(\Gamma\cap T(\overline{K})))\subseteq Z(U)$, and since the reverse inclusion is trivial we conclude that $Z(U)$ is indeed the centralizer of $f(\Gamma\cap T(\overline{K}))\subseteq\Gamma$. In particular, by Corollary~\ref{co:estcentr}\eqref{co:estcentr-ok} we obtain
\begin{equation}\label{eq:boundzu}
\frac{1}{\varphi(d-1)}|\Gamma|^{\dim(Z(U))/d}\leq|\Gamma\cap Z(U)(\overline{K})|\leq\varphi(2)|\Gamma|^{\dim(Z(U))/d}
\end{equation}
with $\varphi(x)$ as in the statement, using the fact that $1\leq\dim(Z(U))\leq 2$ by \eqref{eq:zu}.

Now we search for our $V$. The group $Z(U)$ is nontrivial since $U$ is nilpotent, and every nontrivial $B$-invariant subgroup $V\subseteq Z(U)$ (including $V=Z(U)$ itself) is connected with $V=C_{G}(V)=C_{G}(C_{G}(V))$ by \cite[Prop.~8.4]{LP11}. As $\dim(Z(U))\geq 1$, \eqref{eq:boundzu} and Assumption~\ref{ass-large1} imply that there is some $v\in(\Gamma\cap Z(U)(\overline{K}))\setminus\{e\}$. From now on, fix a $V=V^{(G)}$ of minimal dimension among all possible $B$-invariant subgroups of $Z(U)$ with $|\Gamma\cap V(\overline{K})|>1$, where $B$ also runs among all possible Borel subgroups of $G$. We warn the reader that, although here we just write $V$ instead of $V^{(G)}$, the dependence on $G$ is important in later subsections.

By what we said before, such a $V$ exists, it has $\dim(V)\geq 1$, and there is some $v\in(\Gamma\cap Z(U)(\overline{K}))\setminus\{e\}$ which we can fix. Moreover, $V$ is abelian since $Z(U)$ contains it, which implies also $\dim(V)\leq\min\{2,r\}$. We claim that $v$ sits in the unique open $T$-orbit in $V$: following the discussion in \cite[p.~1135]{LP11}, if $v$ is not in that orbit then either $v=e$ for $\dim(V)=1$, contradicting the choice of $v$, or $v\in U_{\alpha_{\ell}}\cup U_{\alpha_{s}}$ for $\dim(V)=2$, contradicting the minimality of $\dim(V)$. Since $v$ is in the open orbit, by \cite[Prop.~8.4(b)]{LP11} we have $V=C_{G}(v)$: in other words, $V$ is the centralizer of an element of $\Gamma$, and Corollary~\ref{co:estcentr} applies. Now Corollary~\ref{co:centralizer} yields $\deg(V)\leq D$ and Corollary~\ref{co:estcentr}\eqref{co:estcentr-ok} yields
\begin{equation}\label{eq:boundzuv}
\frac{1}{\varphi(d-1)}|\Gamma|^{\dim(V)/d}\leq|\Gamma\cap V(\overline{K})|\leq\varphi(2)|\Gamma|^{\dim(V)/d}
\end{equation}
with $\varphi(x)$ as in the statement, since Assumption~\ref{ass-large1}--\ref{ass-h1} forbids the other cases. Combining again the above with Assumption~\ref{ass-large1}, in which the bound is larger than $\varphi(d-1)^{5d}$, we have $|\Gamma\cap V(\overline{K})|\geq\varphi(d-1)^{4}$. Recall now that $\mathrm{char}(K)=p$, and consider the ring $\mathbb{F}_{p}[N_{\Gamma}(V)]$: by \cite[Prop.~8.4(a)-8.15]{LP11}, it is in fact a finite field $\mathbb{F}_{q}$ and at the same time a finite $\mathbb{F}_{p}$-subalgebra of $K^{\dim(V)}$, and $\Gamma\cap V(\overline{K})$ is a nontrivial finite-dimensional $\mathbb{F}_{q}$-vector space. Then we plug \cite[Thm.~6.6]{LP11} and the bound $|\Gamma\cap V(\overline{K})|\geq\varphi(d-1)^{4}$ into \cite[Lem.~8.16]{LP11}, and we obtain that the dimension of $\Gamma\cap V(\overline{K})$ as a vector space over $\mathbb{F}_{q}$ is $1$. The various claims on $\mathbb{F}_{q}$ and $V$ follow from this conclusion and from \eqref{eq:boundzuv}.
\end{proof}

\subsection{Finding the simple group $[G^{F},G^{F}]$}

Now that we have identified the correct field $\mathbb{F}_{q}$, a copy of which occurs inside $\Gamma$ as a subgroup $\Gamma\cap V(\overline{K})$ of minimal unipotent elements, we must show that $\mathbb{F}_{q}$ is a good choice throughout the whole $\Gamma$. More rigorously, as said before, our goal is now to prove that $[G^{F},G^{F}]\leq\Gamma\leq G^{F}$ for some Steinberg endomorphism $F$ with respect to $\mathbb{F}_{q}$; additionally, $[G^{F},G^{F}]$ shall be a finite simple group of Lie type.

The procedure is covered in \cite[\S\S 9--10--11]{LP11}. In this subsection we follow the theoretical proof with just enough details to allow the reader to understand the process, and focus instead on the quantitative bookkeeping. Overall, the results here are arranged differently than in \cite{LP11}. The correspondence is in general terms as follows:
\medskip
\begin{center}
\begin{longtable}{ccc}
\textbf{Larsen--Pink:} & & \textbf{This paper:} \\ \\
\begin{tabular}{c}Thm.~9.1, basic case, \\ proved in \S 10\end{tabular} & $\longleftrightarrow$ & \begin{tabular}{c}Prop.~\ref{pr:91base}, under \\ Assumption~\ref{ass-large1}--\ref{ass-h1}\end{tabular} \\
$\Big\Downarrow$ & & \\
\begin{tabular}{c}Thm.~0.5, basic case, \\ proved in \S 9\end{tabular} & & {\color{white}(Prop.~5.5)} $\Big\Downarrow$ (Prop.~\ref{pr:9105}) \\
$\Big\Downarrow$ & & \\
\begin{tabular}{c}Thm.~9.1, general case, \\ proved in \S 11\end{tabular} & $\longleftrightarrow$ & \begin{tabular}{c}Prop.~\ref{pr:91all}, under \\ Assumption~\ref{ass-large2}--\ref{ass-h2}\end{tabular} \\
$\Big\Downarrow$ & & {\color{white}(Prop.~5.5)} $\Big\Downarrow$ (Prop.~\ref{pr:9105}) \\
\begin{tabular}{c}Thm.~0.5, general case, \\ proved in \S 9\end{tabular} & $\longleftrightarrow$ & \begin{tabular}{c}Thm.~\ref{th:05}, under \\ Assumption~\ref{ass-large2}--\ref{ass-h2}\end{tabular} 
\end{longtable}
\end{center}
\medskip
The fact that the basic case of \cite[Thm.~9.1]{LP11} is used to prove the general case of the same result is the reason why we have two versions of the quantitative assumptions.

We start by showing that, under the weaker quantitative Assumption~\ref{ass-large1}--\ref{ass-h1} and the stronger condition $\dim(V)=r$, we can build a representation of $G$ onto a module that, up to changing the base field from $K$ to $\mathbb{F}_{q}$, is also $\Gamma$-invariant. One can imagine that this fact gets us close to saying $\Gamma=G(\mathbb{F}_{q})$; the last statement is too strong, but it shall be relaxed to saying $[G^{F},G^{F}]\leq\Gamma\leq G^{F}$ instead.

\begin{proposition}\label{pr:91base}
Let $G,n,r,K,d,D,\iota,\Gamma$ be as in Assumption~\ref{ass-large1}--\ref{ass-h1}. Since Proposition~\ref{pr:findfield} holds, let $p,q=p^{e},V=V^{(G)}$ be as given therein. Assume that $\dim(V)=r$.

Then, there is a $K^{\dim(V)}$-module $M$, there is a nontrivial representation $\sigma:G\rightarrow\mathrm{GL}(M)$ with $\mathrm{mdeg}(\sigma)\leq d^{2}$ ($\sigma$ defined over $K$), and there is an $\mathbb{F}_{q}$-submodule $M_{0}$ of $M$ with $M_{0}\otimes_{\mathbb{F}_{q}}K^{\dim(V)}\simeq M$ and $\sigma(\gamma)(M_{0})=M_{0}$ for every $\gamma\in\Gamma$.
\end{proposition}

\begin{proof}
The proof follows chiefly \cite[\S 10]{LP11}.

The process of constructing $V$ involves fixing $B,U,T$, respectively a Borel subgroup, its unipotent radical, and a maximal torus $T$ satisfying \eqref{eq:but}, so let them be fixed here as well. Since $\dim(V)=r$ we must have $\dim(Z(U))=r$ (by construction $V\leq Z(U)\leq T$), which can happen only for $r=1,2$ by \eqref{eq:zu}. If $\Lambda\subseteq\Gamma$ is the set of elements conjugate to an element of $(\Gamma\cap T(\overline{K}))^{\mathrm{rss}}$, then
\begin{equation}\label{eq:lambdagamma}
|\Lambda|=\frac{|(\Gamma\cap T(\overline{K}))^{\mathrm{rss}}|}{|\Gamma\cap T(\overline{K})|}\cdot\frac{1}{[N_{\Gamma}(T):\Gamma\cap T(\overline{K})]}\cdot|\Gamma|\geq\left(1-\frac{1}{2(2r)^{r}}\right)\cdot\frac{1}{|\mathcal{W}|}\cdot|\Gamma|\geq\frac{|\Gamma|}{2|\mathcal{W}|}
\end{equation}
using \eqref{eq:rss-mts} on $\Theta=\Gamma\cap T(\overline{K})\leq\Gamma\cap C_{G}(\Lambda)^{\mathrm{o}}(\overline{K})$.

Let $\rho$ be given as in \eqref{eq:rho}. We know that $\mathrm{Tr}(\rho(\Gamma))\subseteq K^{r}$ by definition of $\rho$, and that by Proposition~\ref{pr:findfield} $\mathbb{F}_{q}$ is a subalgebra of $K^{r}$ and isomorphic to $\mathbb{F}_{p}[N_{\Gamma}(V)]$ as a ring. We want to prove that $\mathrm{Tr}(\rho(\Lambda))\subseteq\mathbb{F}_{q}$. The case $r=1$ is immediate (see \cite[p.~1143]{LP11}), so assume $r=2$.

By \eqref{eq:but} and the fact that $\Gamma\cap U(\overline{K})$ acts trivially on $V\subseteq Z(U)$, $\Gamma\cap T(\overline{K})$ acts faithfully by conjugation on $V$ and therefore maps isomorphically to a subgroup of $\mathbb{F}_{q}^{*}=(\mathbb{F}_{p}[N_{\Gamma}(V)])^{*}$. Thus $\Gamma\cap T(\overline{K})$ is cyclic, generated by some element $\gamma$ that maps to some $(x,x^{p^{f}})\in\mathbb{F}_{q}^{*}$ (recall that $\mathbb{F}_{q}$ is also a subalgebra of $K^{2}$). Fix such $\gamma,x,f$.
Under Assumption~\ref{ass-large1}--\ref{ass-h1}, one can prove that $2f+1=e$: this is not a surprise, since the cases with $\dim(Z(U))=r=2$ are those of type $B_{2},G_{2}$ with characteristic $p=2,3$ respectively, and the automorphisms $(\cdot)^{p^{f}}$ with $q=p^{2f+1}$ provide the twists that yield Suzuki--Ree groups. The proof is contained in \cite[Lemmas~10.5 to 10.9]{LP11}; here we just sum up the computational details.

There exists some $u\in\Gamma^{\mathrm{run}}$ by Proposition~\ref{pr:run}. Then we can use \eqref{eq:boundtlow} for $r=2$ and Assumption~\ref{ass-large1} to conclude that
\begin{equation*}
|\Gamma\cap T(\overline{K})|\geq\frac{|\Gamma|^{\frac{2}{d}}}{(2dD(\iota+1))^{(d-2)(d+1)d^{d-2}}}\cdot\frac{1}{(2dD)^{d^{2}}(2dD(\iota+1)^{d-2})^{d^{d-2}}}\cdot\frac{1}{16}>1.
\end{equation*}
The subgroup $T'\leq T$ yielding a scalar action on $U/[U,U]$ has dimension $1$, so the bound above and Theorem~\ref{th:dimest} imply that there is some element $\gamma'\in\Gamma\cap T(\overline{K})$ whose action on $U/[U,U]$ is non-scalar. We also need the size of the subgroup of $\mathbb{F}_{q}^{*}$ to which $\Gamma\cap T(\overline{K})$ maps isomorphically to be at least $(q-1)/\varphi(d-1)^{2}$ (true by Assumption~\ref{ass-large1}--\ref{ass-h1} and \cite[Thm.~6.6]{LP11}), and we need the bound $(q-1)/\varphi(d-1)^{2}\geq 2q^{3/4}$
(true by Assumption~\ref{ass-large1} and Proposition~\ref{pr:findfield}). The existence of $u$ and $\gamma'$, the minimality of $V$ as defined inside the proof of Proposition~\ref{pr:findfield}, and the two bounds above force the equality $2f+1=e$. A case-by-case analysis for the possible groups $G$ and the corresponding roots then gives $\mathrm{Tr}(\rho_{s}(\gamma^{i}))=\mathrm{Tr}(\rho_{\ell}(\gamma^{i}))^{p^{f}}$ for all $i$, yielding the case $r=2$ of the claim $\mathrm{Tr}(\rho(\Lambda))\subseteq\mathbb{F}_{q}$.

It is time to define the objects that we look for. By Proposition~\ref{pr:degrho}, $\rho$ is a representation over $K$ of the group $G$ in a space of dimension $\leq d$, and it has $\mathrm{mdeg}(\rho)\leq d$; the same can be said for $\rho$ seen as a representation over $K^{\dim(V)}$. Let $M\leq\mathrm{Mat}_{d}(K^{\dim(V)})$ be the ring of $K^{\dim(V)}$-linear transformations of the representation space of $\rho$. Call $X\subseteq G^{d^{2}}$ the set of $(g_{i})_{i\leq d^{2}}$ for which $(\rho(g_{i}))_{i\leq d^{2}}$ does not span $M$ as a $K^{\dim(V)}$-vector subspace of $(K^{\dim(V)})^{d^{2}}$. The set $X$ is a proper subvariety of $G^{d^{2}}$ (thus of dimension $<d^{3}$) defined as the set of zeros of the $(\dim(M)\times\dim(M))$-minors of the matrix whose rows are the $\rho(g_{i})$. By Corollary~\ref{co:inters}\eqref{co:inters-deg} we have
\begin{equation*}
\deg(X)\leq(\mathrm{mdeg}(\rho)\dim(M))^{d^{3}}\leq d^{3d^{3}},
\end{equation*}
so Corollary~\ref{co:gkbaddim} gives $|\Gamma^{d^{2}}\cap X(\overline{K})|\leq d^{2}(2d^{3d^{3}+1})^{d^{d-1}}|\Gamma|^{d^{2}-\frac{1}{d}}$. On the other hand, if we call $\Omega\subseteq\Gamma^{d^{2}}$ the set of $(\gamma_{i})_{i\leq d^{2}}$ for which
\begin{equation*}
\left|\bigcap_{i=1}^{d^{2}}\gamma_{i}^{-1}\Lambda\right|\geq\frac{|\Gamma|}{2(2|\mathcal{W}|)^{d^{2}}},
\end{equation*}
then by \eqref{eq:lambdagamma} and \cite[Lem.~10.10]{LP11} we also have $|\Omega|\geq|\Gamma|^{d^{2}}/2(2|\mathcal{W}|)^{d^{2}}$.
By \eqref{eq:weilg}, Assumption~\ref{ass-large1}, and the two bounds above, there must be some $(\gamma_{i})_{i\leq d^{2}}\in\Omega\setminus X$. Fix such a tuple, fix $\gamma_{0}\in\bigcap_{i=1}^{d^{2}}\gamma_{i}^{-1}\Lambda$, and define
\begin{equation*}
M_{0}:=\{m\in M:\mathrm{Tr}(\rho(\gamma_{i}\gamma_{0})m)\in\mathbb{F}_{q} \ \ (1\leq i\leq d^{2})\}.
\end{equation*}
This is an $\mathbb{F}_{q}$-vector space contained in $M$ that by definition of $X$ spans the whole $M$ over $K^{\dim(V)}$. Now, \cite[Lem.~10.12]{LP11} shows that $|\Gamma\cap\rho^{-1}(M_{0})(\overline{K})|\geq|\Gamma|/2(2|\mathcal{W}|)^{d^{2}}$. In turn, this implies that the left stabilizer $\Delta:=\{\gamma\in\Gamma:\rho(\gamma)M_{0}=M_{0}\}$ is large, in the sense that $[\Gamma:\Delta]\leq 4(2|\mathcal{W}|)^{d^{2}}$: if it were not, for some left $\Gamma$-translates $M'_{0},M''_{0}$ of $M_{0}$ we would have a large intersection $|\Gamma\cap\rho^{-1}(M'_{0}\cap M''_{0})(\overline{K})|$, meaning at least $|\Gamma|/(4(2|\mathcal{W}|)^{d^{2}})^{2}$ in size (see the proof of \cite[Lem.~10.13]{LP11} for details), but this is not possible because every proper submodule $N\subsetneq M$ must have
\begin{align}\label{eq:nupbound}
\deg(\rho^{-1}(N)) & \leq Dd^{d}, & |\Gamma\cap\rho^{-1}(N)(\overline{K})| & \leq(2Dd^{d+1})^{d^{d}}|\Gamma|^{1-\frac{1}{d}}
\end{align}
by Lemma~\ref{le:zarimdeg}\eqref{le:zarimdeg-pre} and Theorem~\ref{th:dimest}\eqref{th:dimest-ok},
contradicting the previous lower bound by \eqref{eq:weilg} and Assumption~\ref{ass-large1}. Let $\sigma:G\rightarrow\mathrm{GL}(M)$ be the representation defined by
\begin{equation}\label{eq:defsigma}
\sigma(g)(m)=\rho(g)m\rho(g)^{-1}.
\end{equation}
By definition of $\rho$, we have $\mathrm{mdeg}(\sigma)\leq d^{2}$ (as $\rho$ comes from the adjoint representation $\mathrm{mdeg}(\sigma)$ can be shown to be linear in $d$, but the improvement is negligible), and $\sigma$ is defined over $K$. We need to show that $M_{0}=\sigma(\gamma)(M_{0})$ for every $\gamma\in\Gamma$. We know that $\rho(\Delta)\subseteq M_{0}$, and the large index of $\Delta$ implies that
\begin{equation*}
|\Gamma\cap\rho^{-1}(M_{0}\cap\sigma(\gamma)(M_{0}))|\geq|\Delta\cap\gamma\Delta\gamma^{-1}|\geq\left(\frac{1}{4(2|\mathcal{W}|)^{d^{2}}}\right)^{2}|\Gamma|,
\end{equation*}
while if $M_{0}\neq\sigma(\gamma)(M_{0})$ we would get an upper bound on $|\Gamma\cap\rho^{-1}(M_{0}\cap\sigma(\gamma)(M_{0}))(\overline{K})|$ like in \eqref{eq:nupbound},
contradicting again Assumption~\ref{ass-large1}. Therefore $M_{0}=\sigma(\gamma)(M_{0})$, proving the result.
\end{proof}

From the $\Gamma$-invariant module $M_{0}$ we build the desired Steinberg endomorphism in Proposition~\ref{pr:9105} below. We do not assume any condition on $\dim(V)$, so that we may be able to apply the result again later in the paper.

\begin{proposition}\label{pr:9105}
Let $G,n,r,K,d,D,\iota,\Gamma$ be as in Assumption~\ref{ass-large1}--\ref{ass-h1}. Since Proposition~\ref{pr:findfield} holds, let $p,q=p^{e},V=V^{(G)}$ be as given therein. Assume that there exist $\sigma,M,M_{0}$ satisfying the conclusions of Proposition~\ref{pr:91base} (unlike in that result, we do not necessarily assume that $\dim(V)=r$).

Then, there is a Steinberg endomorphism $F:\mathrm{Aut}_{K}(M)\rightarrow\mathrm{Aut}_{K}(M)$ (where $M$ is seen as a module over $K$, instead of over $K^{\dim(V)}$) such that $[G^{F},G^{F}]\leq\Gamma\leq G^{F}$ with $[G^{F},G^{F}]$ simple.
\end{proposition}

\begin{proof}
We follow \cite[\S 9]{LP11}, skipping the theoretical details and focusing on explicit bounds.

By construction $M$ is a $K^{\dim(V)}$-module, where $\dim(V)\in\{1,2\}$ by \eqref{eq:zu}. As $\sigma$ in nontrivial and $G$ is almost simple adjoint, $\mathrm{Ad}_{\sigma(G)}\circ\sigma$ is a totally inseparable isogeny on its image, hence injective (as follows from its definition, see \cite[p.~448]{Pin98}), and since all maps are defined over $K$ then $\sigma$ must be injective on $K$-points.

As $M$ is a module and $\sigma(G)\leq\mathrm{GL}(M)$ (over $K$), we may define our desired Steinberg endomorphism $F:\mathrm{GL}(M)\rightarrow\mathrm{GL}(M)$ as a linear transformation of the vector space $\mathrm{Mat}(M)$, albeit only over $\mathbb{F}_{p}$. As such, $F$ preserves degrees by the original definition, since linear subspaces are sent to linear subspaces, and in particular $\deg(F(\sigma(G)))=\deg(\sigma(G))$.

The details of the choice of $F$ are contained in \cite[p.~1141]{LP11}. If $\dim(V)=1$, we take $F$ to be the Frobenius map $F:\mathrm{Aut}_{K}(M)\rightarrow\mathrm{Aut}_{K}(M)$ with respect to $\mathbb{F}_{q}\subseteq K$. If $\dim(V)=2$, we can write $M=M_{\ell}\oplus M_{s}$ and $\mathrm{Aut}_{K^{2}}(M)=\mathrm{Aut}_{K}(M_{\ell})\times\mathrm{Aut}_{K}(M_{s})$ and take isogenies on the two components so that their product is a map $F$ whose square is the Frobenius map with respect to $\mathbb{F}_{q}\subseteq K^{2}$. By construction, $\sigma^{-1}(F(\sigma(G)))\leq G$, and the conclusion of Proposition~\ref{pr:91base} gives $F(\sigma(\gamma))=\sigma(\gamma)$ for all $\gamma\in\Gamma$, meaning that $\Gamma\subseteq\sigma^{-1}(F(\sigma(G)))$.

The bound $\mathrm{mdeg}(\sigma)\leq d^{2}$, the degree-preserving property of $F$, and Lemma~\ref{le:zarimdeg}\eqref{le:zarimdeg-im}--\eqref{le:zarimdeg-pre} give
\begin{equation*}
\deg(\sigma^{-1}(F(\sigma(G))))\leq Dd^{2d}\deg(F(\sigma(G)))=Dd^{2d}\deg(\sigma(G))\leq D^{2}d^{4d}.
\end{equation*}
Hence, Assumption~\ref{ass-h1} forces $\sigma^{-1}(F(\sigma(G)))=G$. The map $F$ becomes naturally an isogeny $F:G\rightarrow G$ (see \cite[Lem.~9.4]{LP11}, which uses \cite[Thm.~1.7]{Pin98}) for which $F^{\dim(V)}$ is the Frobenius map with respect to $\mathbb{F}_{q}$. On one hand we have $\Gamma\subseteq G^{F}$, using the injectivity of $\sigma$ on $K$-points. On the other hand,
Assumption~\ref{ass-large1} and \cite[Thm.~3.4]{LP11} imply that $[G^{F},G^{F}]$ is simple with index $\leq r+1$ inside $G^{F}$ by \eqref{eq:indexcartan}, and that
\begin{equation}\label{eq:gfbounds}
\frac{|\Gamma|}{2\varphi(d-1)^{\frac{d}{\dim(V)}}}\leq(q^{\frac{1}{\dim(V)}}-1)^{d}\leq|G^{F}|\leq q^{\frac{d}{\dim(V)}}\leq\varphi(2)^{\frac{d}{\dim(V)}}|\Gamma|
\end{equation}
by Proposition~\ref{pr:findfield}. Let $H_{1}:=\Gamma\cap[G^{F},G^{F}]$, and let $H_{2}$ be the normal core of $H_{1}$ inside $[G^{F},G^{F}]$. The upper bound in \eqref{eq:gfbounds} gives
\begin{align}
|[G^{F},G^{F}]/H_{2}| & \leq[[G^{F},G^{F}]:H_{1}]!\leq[G^{F}:\Gamma]!\leq\left(\varphi(2)^{\frac{d}{\dim(V)}}\right)! \nonumber \\
 & \leq\left((2dD(\iota+1))^{2(d+1)d^{3}}\right)!\leq(2dD\iota)^{(2dD\iota)^{5d^{4}}}, \label{eq:gfgfup}
\end{align}
and the lower bound of \eqref{eq:gfbounds} gives
\begin{equation}\label{eq:gfgfdown}
[G^{F},G^{F}]\geq\frac{|G^{F}|}{r+1}\geq\frac{|\Gamma|}{2(r+1)(2dD(\iota+1))^{d^{d+1}}}.
\end{equation}
Putting together \eqref{eq:gfgfup}, \eqref{eq:gfgfdown}, Assumption~\ref{ass-large1}, and the simplicity of $[G^{F},G^{F}]$, we are forced to have $H_{2}=[G^{F},G^{F}]$, concluding that $\Gamma\supseteq[G^{F},G^{F}]$.
\end{proof}

Next we prove the existence of a $\Gamma$-invariant module $M_{0}$ as in Proposition~\ref{pr:91base}, this time without the assumption $\dim(V)=r$. The construction relies on finding a smaller subgroup $H\leq G$ for which $\dim(V^{(H)})=\mathrm{rk}(H)$, and use the field and the representation resulting from $H$ to build the representation of $G$. Passing from $G$ to $H$ worsens the conditions on $\Gamma$, hence we need to impose the stronger Assumption~\ref{ass-large2}--\ref{ass-h2} on $\Gamma\leq G(\overline{K})$ to ensure that a second finite group $\Delta\leq H(\overline{K})$ related to $\Gamma$ satisfies Assumption~\ref{ass-large1}--\ref{ass-h1}.

\begin{proposition}\label{pr:91all}
Let $G,n,r,K,d,D,\iota,\Gamma$ be as in Assumption~\ref{ass-large2}--\ref{ass-h2}. Since Proposition~\ref{pr:findfield} holds, let $p,q=p^{e},V=V^{(G)}$ be as given therein.

Then, there is a $K^{\dim(V)}$-module $M$, there is a nontrivial representation $\sigma:G\rightarrow\mathrm{GL}(M)$ with $\mathrm{mdeg}(\sigma)\leq d^{2}$ ($\sigma$ defined over $K$), and there is an $\mathbb{F}_{q}$-submodule $M_{0}$ of $M$ with $M_{0}\otimes_{\mathbb{F}_{q}}K^{\dim(V)}\simeq M$ and $\sigma(\gamma)(M_{0})=M_{0}$ for every $\gamma\in\Gamma$.
\end{proposition}

\begin{proof}
If $\dim(V)=r$ we are done by Proposition~\ref{pr:91base}, so we may assume $\dim(V)<r$. We follow \cite[\S 11]{LP11}. The first step is to find a suitable $H\leq G$ for which $\dim(V^{(H)})=\mathrm{rk}(H)$ and apply Propositions~\ref{pr:91base}--\ref{pr:9105} to it.

By \cite[\S\S 28.3--28.5]{Hum95b}, we can decompose $G$ as a disjoint union
\begin{equation*}
G=\coprod_{w\in\mathcal{W}}BwB=B\dot{w}B\sqcup\coprod_{w\neq\dot{w}}\overline{BwB},
\end{equation*}
where $B\dot{w}B$ is open and dense. If $X$ is the union over $w\neq\dot{w}$ on the right-hand side, by Lemmas~\ref{le:zarimdeg}\eqref{le:zarimdeg-im}--\ref{le:tudeg}\eqref{le:tudeg-b} and \eqref{eq:weilg} we have
\begin{equation}\label{eq:degbb}
\deg(X)\leq(|\mathcal{W}|-1)\deg(B)^{2}2^{d}\leq((2r)^{r}-1)D^{2}2^{d}.
\end{equation}
Hence, by Assumption~\ref{ass-large2}--\ref{ass-h2} and Theorem~\ref{th:dimest}\eqref{th:dimest-ok} applied to $X$, there is some $\gamma\in\Gamma\cap(B\dot{w}B)(\overline{K})$. From now on, fix such a $\gamma$. Call $H_{(\gamma)}$ the algebraic group generated by $V$ and $\gamma V\gamma^{-1}$: $H_{(\gamma)}$ is connected almost simple of type $A_{1},B_{2},G_{2}$, and is in fact the product of root subgroups normalized by $T$ \cite[Prop.~11.1(a)--(b)]{LP11}. Thus, we can write $H_{(\gamma)}$ as the image of the map from $T\times T$ to $G$ given by $f(t_{1},t_{2})=t_{1}g_{1}t_{1}^{-1}t_{2}g_{2}t_{2}^{-1}$ for some appropriate $g_{1},g_{2}$, which by Lemma~\ref{le:zarimdeg}\eqref{le:zarimdeg-im} gives $\deg(H_{(\gamma)})\leq D^{2}(2\iota+2)^{2}$. Write also $\pi$ for the adjoint representation $\mathrm{Ad}_{H_{(\gamma)}}$ and $H$ for its image $H_{(\gamma)}^{\mathrm{ad}}$: the map $\pi:H_{(\gamma)}\rightarrow H$ has $\mathrm{mdeg}(\pi)\leq\iota+1$ by \eqref{eq:mdegad}, and every fibre has size $\leq 2$ (checking the types case by case). Most notably, $H$ is a connected almost simple adjoint group of a type for which $\dim(V^{(H)})=\mathrm{rk}(H)$, so Propositions~\ref{pr:91base}--\ref{pr:9105} apply to $H$, provided that we also have a suitable finite group inside $H(\overline{K})$.

Define $\Delta:=\pi(\Gamma\cap H_{(\gamma)}(\overline{K}))$. By the bound on fibre sizes and Proposition~\ref{pr:findfield}, we have
\begin{equation*}
|\Delta|\geq\frac{1}{2}|\Gamma\cap H_{(\gamma)}(\overline{K})|\geq\frac{1}{2}|\Gamma\cap V(\overline{K})|\geq\frac{1}{2(2dD(\iota+1))^{d^{d+1}}}|\Gamma|^{\frac{\dim(V)}{d}},
\end{equation*}
and if $\Gamma$ satisfies Assumption~\ref{ass-large2} then $\Delta$ must satisfy Assumption~\ref{ass-large1}. Now, let $L$ be any algebraic subgroup of $H$ that is proper (so $\dim(L)<\dim(H)$ automatically) and that satisfies
\begin{equation}\label{eq:ha1b2g2}
\deg(L)\leq(2\dim(H)\deg(H))^{4\dim(H)}\leq 2^{14000},
\end{equation}
where we used $\dim(H)\leq 10$ and $\deg(H)\leq 2^{7^{3}}$ by Remark~\ref{re:a1b2g2}\eqref{re:a1b2g2-yes}. Since $L$ is proper, by definition of $H_{(\gamma)}$ we cannot have $V\subseteq\pi^{-1}(L)$ and $\gamma V\gamma^{-1}\subseteq\pi^{-1}(L)$ at the same time. Suppose that $V\not\subseteq\pi^{-1}(L)$. Lemma~\ref{le:zarimdeg}\eqref{le:zarimdeg-pre} implies that $\deg(V\cap\pi^{-1}(L))\leq 2^{14000}D(\iota+1)^{2}$. Combining Assumption~\ref{ass-large2} with Theorem~\ref{th:dimest}\eqref{th:dimest-ok} and Proposition~\ref{pr:findfield},
\begin{align*}
|\Gamma\cap(V\cap\pi^{-1}(L))(\overline{K})| & \leq(2d\cdot 2^{14000}D(\iota+1)^{2})^{d^{2}}|\Gamma|^{\frac{\dim(V)-1}{d}} \\
 & <\frac{|\Gamma|^{\frac{\dim(V)}{d}}}{(2dD(\iota+1))^{d^{d+1}}}\leq|\Gamma\cap V(\overline{K})|,
\end{align*}
implying that $\Gamma\cap V(\overline{K})\not\subseteq\pi^{-1}(L)(\overline{K})$. Analogously, $\gamma V\gamma^{-1}\not\subseteq\pi^{-1}(L)$ gives $\Gamma\cap\gamma V\gamma^{-1}(\overline{K})\not\subseteq\pi^{-1}(L)(\overline{K})$. In either case, we obtain $\Delta\not\subseteq L(\overline{K})$, which means that $\Delta$ satisfies Assumption~\ref{ass-h1}. Hence, we can apply Proposition~\ref{pr:9105} to $\Delta$ and $H$, we have $[H^{F},H^{F}]\leq\Delta\leq H^{F}$ for some Steinberg endomorphism with $[H^{F},H^{F}]$ simple.

Finally, by Assumption~\ref{ass-large2} and \eqref{eq:boundzuv} there is some $v\in(\Gamma\cap V(\overline{K}))\setminus\{e\}$, which we fix. If $\gamma$ runs through the elements of $\Gamma\cap (B\dot{w}B)(\overline{K})$ the group $H_{(\gamma)}$ may change, but all such $H_{(\gamma)}$ are conjugate by \cite[Prop.~11.1(c)]{LP11}; therefore $H,F$ do not depend on the choice of $\gamma$. The proof of \cite[Prop.~11.7]{LP11} then shows that $\mathbb{F}_{q}$ and the field underlying the map $F$ are the same field, and that, for $\rho$ as in \eqref{eq:rho}, $\mathrm{Tr}(\rho(v\gamma v\gamma^{-1}))\in\mathbb{F}_{q}$ for all $\gamma\in\Gamma\cap(B\dot{w}B)(\overline{K})$ (and $v$ fixed).

Now we use the result above about traces to define the appropriate $\sigma,M,M_{0}$. Let $\bar{M}$ be the ring of $K^{\dim(V)}$-linear transformations of the representation space of $\rho$, and let $\bar{M}'$ be the smallest $G$-invariant $K^{\dim(V)}$-submodule of $\bar{M}$ containing $\rho(v)$. Just as in the proof of Proposition~\ref{pr:91base}, there is a representation $\bar{\sigma}:G\rightarrow\mathrm{GL}(\bar{M})$ given as in \eqref{eq:defsigma}, which in particular is defined over $K$ and satisfies $\mathrm{mdeg}(\bar{\sigma})\leq d^{2}$. Its restriction $\bar{\sigma}':G\rightarrow\mathrm{GL}(\bar{M}')$ is again defined over $K$, because $\bar{\sigma}$ and $M'$ are, and has $\mathrm{mdeg}(\bar{\sigma}')\leq d^{2}$. Since $\bar{\sigma}(g)$ is conjugation by $\rho(g)$ it preserves traces, so we may take the quotient $M:=\bar{M}'/(\bar{M}'\cap\bar{M}'^{\perp})$ (the orthogonal complement is taken with respect to the trace form) and still obtain a representation $\sigma:G\rightarrow\mathrm{GL}(M)$ with $\mathrm{mdeg}(\sigma)\leq d^{2}$. Moreover $\sigma$ is nontrivial by \cite[Prop.~11.5, (11.9)]{LP11}, and it is defined over $K$ because $\bar{M}'^{\perp}$ is. Finally, let $m_{0}\in M$ be the element corresponding to $\rho(v)\in\bar{M}'$, and call $M_{0}$ the $\mathbb{F}_{q}$-submodule generated by the orbit $O_{\Gamma}(m_{0})$. By construction, $\sigma(\gamma)(M_{0})=M_{0}$ for any $\gamma\in\Gamma$.

It remains to prove that $M_{0}\otimes_{\mathbb{F}_{q}}K^{\dim(V)}\simeq M$. See the proof of \cite[Lem.~11.12]{LP11}; we only present the computational details. Let $N$ be a proper $K^{\dim(V)}$-submodule of $M$, and define
\begin{equation*}
X(N):=\{g\in G:\sigma(g)(m_{0})\in N\}
\end{equation*}
as in \cite[(11.13)]{LP11}. The variety $X(N)$ is proper inside $G$, and it is the preimage of $N$ through the map $\sigma(\cdot)(m_{0})$, so
\begin{equation*}
\deg(X(N))\leq\deg(G)\deg(N)\mdeg(\sigma)^{\dim(N)}\leq Dd^{2d}
\end{equation*}
by Lemma~\ref{le:zarimdeg}\eqref{le:zarimdeg-pre}. If the natural map $M_{0}\otimes_{\mathbb{F}_{q}}K^{\dim(V)}\rightarrow M$ is not surjective then $\Gamma\subseteq X(N)(\overline{K})$ for some $N$, and if it is not injective then
\begin{equation*}
\Gamma\subseteq X(N)(\overline{K})\cup\bigcup_{i=1}^{\ell}\bigcup_{w\neq\dot{w}}(\gamma_{i}BwB)(\overline{K}),
\end{equation*}
where $\ell$ is at most one more than the dimension of $M$ as a $K^{\dim(V)}$-vector space. Recalling the degree bound of \eqref{eq:degbb}, in both cases we obtain that $\Gamma$ is contained in a variety of degree $\leq 2^{d+r+2}r^{r}D^{2}d^{2d}\leq(2dDr)^{2d+r+1}$; applying Lemma~\ref{le:groupvar},
we contradict either Assumption~\ref{ass-large2} or \ref{ass-h2}, so the map must be an isomorphism, and we are done.
\end{proof}

Combining Propositions~\ref{pr:9105} and~\ref{pr:91all}, we reach the finite simple group $[G^{F},G^{F}]$ that we are looking for. Below we write a self-contained statement, but the conditions on $G$ and $\Gamma$ coincide with Assumption~\ref{ass-large2}--\ref{ass-h2}.

\begin{theorem}\label{th:05}
Let $G\leq\mathrm{GL}_{n}$ be a connected almost simple adjoint group of rank $r$ defined over $K$, with $d=\dim(G)$, $D=\deg(G)$, and $\iota=\mathrm{mdeg}(^{-1})$. Let $\Gamma\leq G(\overline{K})$ be finite. Assume the following:
\begin{enumerate}[(a)]
\item\label{th:05-small} $|\Gamma|>(2dDrn\iota)^{(2dDr\iota)^{11d^{4}}}$;
\item\label{th:05-h} $\Gamma\not\leq H(\overline{K})$ for any subgroup $H\lneq G$ with $\dim(H)<d$ and $\deg(H)\leq(2dDr)^{4d^{2}}$.
\end{enumerate}
Then $\mathrm{char}(K)=p>0$, there is some $q=p^{e}$ such that $\mathbb{F}_{q}$ is an $\mathbb{F}_{p}$-subalgebra of either $K$ or $K^{2}$, and there is a Steinberg endomorphism $F:G\rightarrow G$ such that either $F$ or $F^{2}$ is the Frobenius map with respect to $\mathbb{F}_{q}$, with
\begin{equation*}
[G^{F},G^{F}]\leq\Gamma\leq G^{F}
\end{equation*}
and with $[G^{F},G^{F}]$ simple.
\end{theorem}

\begin{proof}
Since the conditions on $G$ and $\Gamma$ are exactly the ones in Assumption~\ref{ass-large2}--\ref{ass-h2}, we obtain Proposition~\ref{pr:91all}. Then we apply Proposition~\ref{pr:9105}, whose hypothesis is the weaker Assumption~\ref{ass-large1}--\ref{ass-h1}, and the result follows.
\end{proof}

\section{Proof of the main theorem}\label{se:final}

The main result of the previous section is that, for $G$ almost simple adjoint and $\Gamma$ sufficiently general, we have $[G^{F},G^{F}]\leq\Gamma\leq G^{F}$ for some Steinberg endomorphism such that $[G^{F},G^{F}]$ is a finite simple group of Lie type. This assertion lies at the the core of \eqref{th:main-lie} in Theorem~\ref{th:main}. In the current section we build the rest of the theorem around this initial core.

The proof relies mainly on a descent process: if $\Gamma$ is not sufficiently general (meaning that it violates either \eqref{th:05-small} or \eqref{th:05-h} in Theorem~\ref{th:05}), then it is trapped in a smaller subgroup, for some definition of ``smaller''; repeating the procedure enough times, eventually the subgroup becomes zero-dimensional, giving a bound on $|\Gamma|$ in terms of $n$ only. Being connected almost simple adjoint is not preserved at every step though, so we need a few more steps in between. If $G$ is an algebraic group, we first pass to a small-index connected subgroup $G^{\mathrm{o}}$ (and the effect on $\Gamma$ is absorbed by \eqref{th:main-small} in the main theorem), then we get to a reductive group via taking quotient by $R_{u}(G)$ (which is absorbed by \eqref{th:main-p}), then to a semisimple group via quotient by $Z(G)$ (which is absorbed by \eqref{th:main-ab}). A semisimple group is a product of almost simple pieces, which gives either a product of finite simple groups as in \eqref{th:main-lie}, up to a small index contributing to \eqref{th:main-small}, or a descent as before. The potential case of $\dim(G)=0$ and bounded $|\Gamma|$ is again dealt with by \eqref{th:main-small}.

Let us specify what parameter we use to track the descent. Let $\{G_{i}\}_{i\in\mathcal{I}(G)}$ be the collection of the connected almost simple adjoint factors of the semisimple group $(G^{\mathrm{o}}/R_{u}(G^{\mathrm{o}}))^{\mathrm{ad}}$; the set $\mathcal{I}(G)$ could be empty, and the $G_{i}$ could be defined over $\overline{K}$ instead of $K$. Define
\begin{equation*}
d_{\mathrm{ad}}(G):=\sum_{i\in\mathcal{I}(G)}\dim(G_{i})=\dim((G^{\mathrm{o}}/R_{u}(G^{\mathrm{o}}))^{\mathrm{ad}}).
\end{equation*}
Trivially $d_{\mathrm{ad}}(G)\leq\dim(G)$, and if $G$ is almost simple then equality holds. It is also easy to show that if $H\leq G$ then $d_{\mathrm{ad}}(H)\leq d_{\mathrm{ad}}(G)$. We set up an induction using $d_{\mathrm{ad}}(G)$.

\begin{lemma}\label{le:badadj}
Let $G\leq\mathrm{GL}_{n}$ be an algebraic group of rank $r$ defined over $K$, with $d=\dim(G)$, $D=\deg(G)$, and $\iota=\mathrm{mdeg}(^{-1})$. Let $\Gamma\leq G(\overline{K})$ be finite.

Assume $\mathcal{I}=\mathcal{I}(G)\neq\emptyset$ (i.e.\ $d_{\mathrm{ad}}(G)\geq 1$). Then at least one of the following holds:
\begin{enumerate}[(a)]
\item\label{le:badadj-h} $\Gamma\leq H(\overline{K})$ for some subgroup $H\lneq G$ with $d_{\mathrm{ad}}(H)<d_{\mathrm{ad}}(G)$ and
\begin{equation*}
\deg(H)\leq(2dn\iota)^{(2dn\iota)^{2^{20}d^{4}r^{2}+2dn^{2}}}D^{d+1};
\end{equation*}
\item\label{le:badadj-ok} $\mathrm{char}(K)=p>0$, and for every $i\in\mathcal{I}$ there is a Steinberg endomorphism $F_{i}:G_{i}\rightarrow G_{i}$ with
\begin{equation*}
[G_{i}^{F_{i}},G_{i}^{F_{i}}]\leq\pi_{i}(\Gamma\cap G^{\mathrm{o}}(\overline{K}))\leq G_{i}^{F_{i}}
\end{equation*}
and with $[G_{i}^{F_{i}},G_{i}^{F_{i}}]$ simple.
\end{enumerate}
\end{lemma}

\begin{proof}
Since $\mathrm{GL}_{1}$ is abelian we may assume $n\geq 2$, or else $\mathcal{I}=\emptyset$. Take $Y=Y(G^{\mathrm{o}})$ and the corresponding $\hat{G},\hat{Y},\lambda,\hat{\beta},m$ as in Lemma~\ref{le:rzdeg}. Since $\hat{G}(\overline{K})\simeq G^{\mathrm{o}}(\overline{K})$, we may consider $\Gamma\cap G^{\mathrm{o}}(\overline{K})\leq\hat{G}(\overline{K})$. The quotient $\hat{G}/\hat{Y}$ (possibly defined over $\overline{K}$) is connected adjoint, so it is isomorphic to the direct product of the $G_{i}$: the isomorphism can be taken to be the adjoint representation $\mathrm{Ad}_{\hat{G}/\hat{Y}}$, whose image sits in $\mathrm{GL}(\mathfrak{x})$ (where $\mathfrak{x}$ is the Lie algebra of $\hat{G}/\hat{Y}$) and which has maximum degree $\leq m+1$ by \eqref{eq:mdegad}. Then, up to a change of basis of $\mathfrak{x}$ (which does not affect the degree), the projection to any almost simple factor has maximum degree $\leq 1$.

For every $i\in\mathcal{I}$, call $\pi_{i}:\hat{G}\rightarrow G_{i}$ the natural epimorphism given by composing the quotient map $\hat{\beta}$, the adjoint representation $\mathrm{Ad}_{\hat{G}/\hat{Y}}$, the change of basis of $\mathfrak{x}$, and the projection to $G_{i}$. By the facts above,
\begin{equation}\label{eq:degsimplefac}
\mathrm{mdeg}(\pi_{i})\leq 2(n^{2}+(\iota+1)d^{d})^{n^{2}}(\iota+1)d^{d}\cdot \left(2^{2(n^{2}+(\iota+1)d^{d})^{n^{2}}}+1\right)\leq(2dn\iota)^{(2dn\iota)^{2dn^{2}}}.
\end{equation}
Furthermore, by Proposition~\ref{pr:adjbounds} we have the following bounds on the parameters of the $G_{i}$ depending only on $d=\dim(G)=\dim(\hat{G})$ and $r=\mathrm{rk}(G)=\mathrm{rk}(\hat{G})$:
\begin{align}\label{eq:parameters}
\mathrm{rk}(G_{i}) & \leq r\leq d, & \dim(G_{i}) & \leq d, & \iota|_{G_{i}} & \leq n|_{G_{i}}\leq d, & \deg(G_{i}) & \leq(2r)^{2^{16}r^{2}}.
\end{align}

For each $i\in\mathcal{I}$, apply Theorem~\ref{th:05} to the finite group $\pi_{i}(\Gamma\cap G^{\mathrm{o}}(\overline{K}))$ contained in $G_{i}(\overline{K})$. If conditions \eqref{th:05-small}--\eqref{th:05-h} are both satisfied for all $i$, we obtain case \eqref{le:badadj-ok} of the present result. Assume then that either \eqref{th:05-small} or \eqref{th:05-h} in Theorem~\ref{th:05} is violated for some $i$. As a matter of fact, we may rewrite case \eqref{th:05-small} to look like an instance of case \eqref{th:05-h}: to do so, interpret $\pi_{i}(\Gamma\cap G^{\mathrm{o}}(\overline{K}))$ as the zero-dimensional algebraic group $H$ defined as the union $\bigcup_{\gamma\in\pi_{i}(\Gamma\cap G^{\mathrm{o}}(\overline{K}))}\{\gamma\}$. Hence, in both cases, $\Gamma\cap G^{\mathrm{o}}(\overline{K})\leq\pi_{i}^{-1}(H)(\overline{K})$ for some algebraic subgroup $H\leq G_{i}$ such that either $\dim(H)<\dim(G_{i})$ and $\deg(H)\leq(2d)^{2^{19}d^{2}r^{2}}$, or $\dim(H)=0$ and $\deg(H)\leq(2d)^{(2d)^{2^{20}d^{4}r^{2}}}$.

Set $P:=\overline{\lambda(\pi_{i}^{-1}(H))}$, which is an algebraic group. As $G^{\mathrm{o}}$ is irreducible we have $P\lneq G^{\mathrm{o}}\leq G$ and $\dim(P)<d$, while $d_{\mathrm{ad}}(H)\leq\dim(H)<\dim(G_{i})=d_{\mathrm{ad}}(G_{i})$ implies that $d_{\mathrm{ad}}(P)<d_{\mathrm{ad}}(G)$. Lemma~\ref{le:zarimdeg}\eqref{le:zarimdeg-im}--\eqref{le:zarimdeg-pre} gives $\deg(P)\leq D\deg(H)\mathrm{mdeg}(\pi_{i})^{d}$. Let $L=\bigcap_{\gamma\in\Gamma}\gamma P\gamma^{-1}$. Again $d_{\mathrm{ad}}(L)\leq d_{\mathrm{ad}}(P)<d_{\mathrm{ad}}(G)$, and Corollary~\ref{co:inters}\eqref{co:inters-deg} yields $\deg(L)\leq\deg(P)^{d}$.

Since $\Gamma$ is finite, $\Gamma L=\bigcup_{\gamma\in\Gamma}\gamma L$ is a variety. Moreover, the normalizer of $L$ in $G$ contains $\Gamma$, so $\Gamma L$ is an algebraic subgroup of $G$ containing $\Gamma$. We have $(\Gamma L)^{\mathrm{o}}=L^{\mathrm{o}}$, giving $d_{\mathrm{ad}}(\Gamma L)=d_{\mathrm{ad}}(L)<d_{\mathrm{ad}}(G)$. It remains to bound $\deg(\Gamma L)$. By definition, $\Gamma\cap L(\overline{K})$ is the normal core of $\Gamma\cap P(\overline{K})=\Gamma\cap G^{\mathrm{o}}(\overline{K})$ inside $\Gamma$, but $\Gamma\cap G^{\mathrm{o}}(\overline{K})\unlhd\Gamma$ already, therefore
\begin{equation*}
|\Gamma L(\overline{K})/L(\overline{K})|=|\Gamma/(\Gamma\cap L(\overline{K}))|=|\Gamma/(\Gamma\cap G^{\mathrm{o}}(\overline{K}))|\leq|G/G^{\mathrm{o}}|\leq D.
\end{equation*}
Hence, using \eqref{eq:degsimplefac},
\begin{equation*}
\deg(\Gamma L)\leq D\deg(L)\leq D^{d+1}\deg(H)^{d}\mathrm{mdeg}(\pi_{i})^{d^{2}}\leq(2dn\iota)^{(2dn\iota)^{2^{20}d^{4}r^{2}+2dn^{2}}}D^{d+1},
\end{equation*}
so the algebraic subgroup $\Gamma L$ satisfies the conditions in \eqref{le:badadj-h}.
\end{proof}

With the induction step of Lemma~\ref{le:badadj} at hand, we can prove the main theorem.

\begin{proof}[Proof of Theorem~\ref{th:main}]
If $\Gamma$ is abelian then the result holds by taking $\Gamma_{2}=\Gamma$ and $\Gamma_{3}$ its Sylow $p$-subgroup. Thus we may assume $\Gamma$ non-abelian, and thus $n\geq 2$.

Start with $G=\mathrm{GL}_{n}$, and apply Lemma~\ref{le:badadj} to it. If case \eqref{le:badadj-h} of the lemma holds, repeat the process with the subgroup of $G$ found in this way, and repeat this step until either $\mathcal{I}(G)=\emptyset$ or we reach case \eqref{le:badadj-ok}. Since $d_{\mathrm{ad}}(G)$ strictly decreases at each step, the process ends in at most $d$ steps. By the bound on $\deg(H)$ inside Lemma~\ref{le:badadj}\eqref{le:badadj-h} and the natural bounds on $\dim(\mathrm{GL}_{n}),\deg(\mathrm{GL}_{n}),\mathrm{rk}(\mathrm{GL}_{n}),\mathrm{mdeg}(^{-1})$ in terms of $n$, at the last step we have $\deg(G)\leq n^{n^{(2^{23}-1)n^{10}}}$.

Let $Y=Y(G^{\mathrm{o}})$ be as in Lemma~\ref{le:rzdeg}. Since $Y\csgp G^{\mathrm{o}}\csgp G$ and $R_{u}(G^{\mathrm{o}})\csgp G^{\mathrm{o}}\csgp G$, we have $Y\csgp G$ and $R_{u}(G^{\mathrm{o}})\csgp G$. Call $\Gamma_{2}:=\Gamma\cap Y(\overline{K})$ and $\Gamma_{3}:=\Gamma\cap R_{u}(G^{\mathrm{o}})(\overline{K})$: we have $\Gamma_{3}\unlhd\Gamma_{2}\unlhd\Gamma$ and $\Gamma_{3}\unlhd\Gamma$.

By construction, $Y/R_{u}(G^{\mathrm{o}})$ is the centre of $G^{\mathrm{o}}/R_{u}(G^{\mathrm{o}})$; in particular, its finite subgroup $\Gamma_{2}/\Gamma_{3}$ is contained in a torus of the reductive group $G^{\mathrm{o}}/R_{u}(G^{\mathrm{o}})$, thus it is abelian of order not divisible by $\mathrm{char}(K)$. Since $R_{u}(G^{\mathrm{o}})$ is unipotent, there is a central normal series whose quotients are isomorphic to algebraic subgroups of $\mathbb{G}_{a}$ \cite[Prop.~14.21]{Mil17}; hence, since $\Gamma_{3}\leq R_{u}(G^{\mathrm{o}})$ is finite, either it is trivial (if $\mathrm{char}(K)=0$) or it is a $p$-group (if $\mathrm{char}(K)=p>0$).

It remains to deal with $\Gamma/\Gamma_{2}$. If $\mathcal{I}(G)=\emptyset$, by definition $G^{\mathrm{o}}/R_{u}(G^{\mathrm{o}})$ is equal to its own centre, so we set $\Gamma_{1}:=\Gamma_{2}$ and the quotient $\Gamma/\Gamma_{2}$ has size $\leq|G/G^{\mathrm{o}}|\leq\deg(G)$. Assume from now on that we have fallen into case \eqref{le:badadj-ok} of Lemma~\ref{le:badadj}. Therefore, $\mathrm{char}(K)=p>0$ and $(\Gamma\cap G^{\mathrm{o}}(\overline{K}))/\Gamma_{2}$ is a subgroup of the nonempty product $R=\prod_{i\in\mathcal{I}}\pi_{i}(\Gamma\cap G^{\mathrm{o}}(\overline{K}))$, for which we know that
\begin{equation*}
\prod_{i\in\mathcal{I}}[G_{i}^{F_{i}},G_{i}^{F_{i}}]\leq R\leq\prod_{i\in\mathcal{I}}G_{i}^{F_{i}}
\end{equation*}
and that each factor $[G_{i}^{F_{i}},G_{i}^{F_{i}}]$ is a finite simple group of Lie type of characteristic $p$. The commutator $[R,R]$ must then be equal to the product of the $[G_{i}^{F_{i}},G_{i}^{F_{i}}]$ and, by a folklore consequence of Goursat's lemma (see \cite[Prop.~3.3]{BS92}), if we call $\Gamma_{1}:=[\Gamma\cap G^{\mathrm{o}}(\overline{K}),\Gamma\cap G^{\mathrm{o}}(\overline{K})]\Gamma_{2}$ then $\Gamma_{1}/\Gamma_{2}$ must be isomorphic to the product of some of the $[G_{i}^{F_{i}},G_{i}^{F_{i}}]$. By construction $\Gamma_{1}\unlhd\Gamma$, since $\Gamma\cap G^{\mathrm{o}}$, its commutator, and $\Gamma_{2}$ are all normal in $\Gamma$. Finally,
\begin{equation*}
|\Gamma/\Gamma_{1}|\leq|G/G^{\mathrm{o}}|\prod_{i\in\mathcal{I}}|G_{i}^{F_{i}}/[G_{i}^{F_{i}},G_{i}^{F_{i}}]|\leq\deg(G)(r+1)^{d}\leq n^{n^{2^{23}n^{10}}}
\end{equation*}
by \eqref{eq:indexcartan}.
\end{proof}

\section*{Acknowledgements}

The second author was funded by a Young Researcher Fellowship from the HUN-REN Alfr\'ed R\'enyi Institute of Mathematics, and by the Leibniz Fellowship 2405p from the Mathematisches Forschungsinstitut Oberwolfach (MFO).

The authors would like to thank the MFO for providing a wonderful environment for the development of this article.

\nocite{}
\bibliographystyle{abbrv}
\bibliography{Biblio}

\end{document}